%% file: ms.tex
\documentclass[a4paper,american,reqno]{amsart}

\newif\ifPreprint \Preprinttrue
\newif\ifSubmission \Submissionfalse

\input{setup-preprint}
\bibliography{literature.bib}

\usepackage[foot]{amsaddr}

\begin{document}

\title[Norm-induced Cuts]{Norm-induced Cuts: \\ Outer Approximation for Lipschitzian Constraint Functions}

\author[A. Gö{\ss}, A. Martin, S. Pokutta, K. Sharma]%
{Adrian Gö{\ss$^\ast$}\orcidlink{https://orcid.org/0009-0002-7144-8657}, 
	Alexander~Martin\orcidlink{https://orcid.org/0000-0001-7602- 3653}, 
	Sebastian~Pokutta\orcidlink{https://orcid.org/0000-0001-7365-3000}, 
	Kartikey~Sharma\orcidlink{https://orcid.org/0000-0001-6736-4827}}

\address{$^\ast$Corresponding author}
\address[A. Gö{\ss}, A. Martin]{%
  University of Technology Nuremberg,
  Analytics \& Optimization Lab,
  Dr.-Luise-Herzberg-Str.~4,
  90461~Nürnberg,
  Germany}
\email{$\{$adrian.goess,alexander.martin$\}$@utn.de,pokutta@zib.de,\\kartikeyrinwa@gmail.com}

\address[S. Pokutta, K. Sharma]{%
	Zuse Institute Berlin,
	Mathematical Algorithmic Intelligence,
	Takustr.~7,
	14195~Berlin,
	Germany}

\begin{abstract}
  \input{abstract}
\end{abstract}

\keywords{\input{keywords}}
\subjclass[2010]{\input{msc2010}}

\maketitle\
\input{introduction}

\input{algorithm}
\input{convergence}
\input{termination}
\input{examples}

\input{conclusion}

\input{acronyms}

\input{acknowledgements}

\printbibliography
\newpage

\input{appendix}

\end{document}

%% file: setup-preprint.tex
\usepackage{babel}
\usepackage[utf8]{inputenc}
\usepackage[T1]{fontenc}
\usepackage{booktabs}
\usepackage{xspace}
\usepackage[binary-units=true]{siunitx}
\usepackage{etoolbox}
\robustify\bfseries
\usepackage{tabularx}
\usepackage[inline]{enumitem}

\usepackage{lmodern}
\usepackage{amssymb}
\usepackage{mathtools}
\usepackage{algorithm}
\floatname{algorithm}{Method}
\usepackage{algorithmic}
\usepackage{tikz,pgfplots}
\usetikzlibrary{intersections}
\usetikzlibrary{patterns}
\usetikzlibrary{shapes.geometric}
\usetikzlibrary{datavisualization}

\usepackage{placeins}
\usepackage{relsize}
\usepackage{algorithm,algorithmic}
\usepackage{caption}
\usepackage{subcaption}
\usepackage{todonotes}
\usepackage{csquotes}

\usepackage{microtype}
\usepackage{textcomp}
\usepackage[style = numeric-comp,
            maxbibnames = 100,
            maxcitenames = 2,
            giveninits = true,
            isbn = false,
            backend = biber]{biblatex}
\usepackage[colorlinks,
            citecolor=blue,
            urlcolor=blue,
            linkcolor=blue]{hyperref} % should always be the last package
\usepackage{orcidlink}
\usepackage{cleveref}
\usepackage[nolist]{acronym}
\Crefname{algorithm}{\ac{nic}}{}
\allowdisplaybreaks
\makeatletter
\patchcmd{\@settitle}{\uppercasenonmath\@title}{\scshape\large}{}{}
\patchcmd{\@setauthors}{\MakeUppercase}{\scshape\normalsize}{}{}
\makeatother

\vfuzz \hfuzz
\input{macros}

%% file: macros.tex
\newcommand{\abbr}[1][abbrev]{#1.\xspace}

\newcommand{\eg}{\abbr[e.g]}

\newcommand{\ie}{\abbr[i.e]}

\newcommand{\wrt}{\abbr[w.r.t]}

\newcommand{\define}{\mathrel{{\mathop:}{=}}}

\newcommand{\field}{\mathbb}
\newcommand{\naturals}{\field{N}}
\newcommand{\rationals}{\field{Q}}
\newcommand{\reals}{\field{R}}
\newcommand{\integers}{\field{Z}}
\newcommand{\NN}{\naturals}
\newcommand{\QQ}{\rationals}
\newcommand{\RR}{\reals}

\makeatletter
\newcommand{\fcdot}{\,\cdot\,}
\newcommand{\fcarg}[1]{\def\fc@rg{#1}\ifx\fc@rg\empty\fcdot\else\fc@rg\fi}
\makeatother
\newcommand{\abs}[1]{\lvert\fcarg{#1}\rvert}

\newcommand{\norm}[2][]{\lVert\fcarg{#2}\rVert\ifx#1\empty\else_{#1}\fi}
\newcommand{\Norm}[2][]{\left\lVert#2\right\rVert\ifx#1\empty\else_{#1}\fi}

\newcommand{\sset}[1]{\{#1\}}
\newcommand{\SSet}[1]{\left\{#1\right\}}
\newcommand{\defset}[3][\defsep]{\sset{ #2#1#3 }}
\newcommand{\Defset}[3][\, \middle| \,]{\SSet{ #2#1#3 }}

\newcommand{\eps}{\varepsilon}

\newtheorem{theorem}{Theorem}[section]
\newtheorem{lemma}[theorem]{Lemma}

\newtheorem{definition}[theorem]{Definition}
\newtheorem{remark}[theorem]{Remark}
\newtheorem{example}[theorem]{Example}

\crefname{assumption}{assumption}{assumptions}

\newcommand{\vect}[1]{\mathbf{#1}} 

\newcommand{\vx}{\vect{x}}
\newcommand{\vy}{\vect{y}}
\newcommand{\va}{\vect{a}}
\newcommand{\vb}{\vect{b}}

\newcommand{\vr}{\vect{r}}
\newcommand{\sol}{\vx^*}

\newcommand{\veczero}{\textup{\textbf{0}}}
\newcommand{\T}{\mathsf{^T}}
\newcommand{\incumbent}{\vx^k}

\newcommand{\subincumbent}{\vx^{i_j}}
\newcommand{\finseries}{(\vx^i)_{i = 0,\dots,k-1}}
\newcommand{\infseries}{(\vx^i)_{i \in \NN_0}}

\newcommand{\infsubseries}{(\vx^{i_j})_{j \in \NN_0}}
\newcommand{\cutlhs}[1]{\frac{\Vert \vr(\vx^{#1})_+ \Vert}{L}}
\newcommand{\packnum}{M(\delta/L)}
\newcommand{\vol}{\mathrm{vol}}

\makeatletter
\DeclareCaptionLabelFormat{numberless}{\fname@algorithm}
\makeatother

\DeclareUnicodeCharacter{2014}{\xxx}
\DeclareRobustCommand\xxx{%
	\xxx\thinspace}

%% file: abstract.tex
In this paper, we consider a finite-dimensional optimization problem minimizing a continuous objective on a compact domain subject to a multi-dimensional constraint function. 
For the latter, we assume
 the availability of a global Lipschitz constant.
In recent literature, methods based on non-convex outer approximation are proposed for tackling one-dimensional equality constraints
that are Lipschitz with respect to the maximum norm.
To the best of our knowledge, however, there does not exist a non-convex outer approximation method for a general problem class. 
We introduce a meta-level solution framework to solve such problems and tackle the underlying theoretical foundations.
Considering the feasible domain without the constraint function as manageable, our method relaxes the multidimensional constraint and iteratively refines the feasible region by means of norm-induced cuts, 
relying on an oracle for the resulting subproblems.
We show the method's correctness and investigate the problem complexity.
In order to account for discussions about functionality, limits, and extensions, we present computational examples including illustrations.

%% file: keywords.tex
Global optimization,
Lipschitz optimization,
Outer Approximation,
Derivative-free optimization
\hspace{-2mm}

%% file: msc2010.tex
90C26, 
90C30, 
90C56, 
90C60

%% file: introduction.tex
\section{Introduction}
\label{sec:introduction}

We present an algorithm framework addressing the following general form of an optimization problem.
\begin{equation}
	\label{original}
	\begin{aligned}
		\min_{\vx}\ & f(\vx) \\
		\text{s.t.}\ & \vr(\vx) \leq \veczero, \\
		& \vx \in \Omega.
	\end{aligned}\tag{P}
\end{equation}
We assume $\Omega$ to be a non-empty, compact subset of the finite-dimensional vector space $X$ over a field $F$, \ie, $\Omega \subset X \cong F^n$, $n \in \naturals$. 
The space $X$ is associated with a norm $\norm[X]{\, \cdot\,}$, leaving $(X, \norm[X]{\,\cdot\,})$ as a normed space. 
The function $\vr : \Omega \rightarrow \reals^m$, $m \in \naturals$, 
is assumed to be Lipschitz continuous on $\Omega$ with a known Lipschitz constant $L > 0$, \ie, for any $\vx,\, \vy \in \Omega$, it holds that
\begin{equation*}
	\norm[\reals^m]{\vr(\vx) - \vr(\vy) } \leq L \norm[X]{\vx - \vy},
\end{equation*}
with $\norm[\reals^m]{\,\cdot\,}$ denoting a monotonous norm in $\reals^m$. 
We mention that $\veczero \in \reals^m$ denotes the vector of $m$ zeros and the inequality in \labelcref{original} is to be interpreted component-wise.
Lastly, the objective function $f: \Omega \rightarrow \reals$ is assumed to be continuous on $\Omega$.

Note that the minimal value of a Lipschitz constant~$L$ is directly connected to the specific norms. 
Further, the posed problem statement requires no explicit form of $\vr$ per se, but only 
the evaluation of $\vr$ at any point $\vx \in \Omega$. 
In addition, problem~\labelcref{original} is presented for an arbitrary field $F$, as the general formulation of our method allows for it.
As mentioned below, the field can even be substituted (partially) by a ring such as the integers $\integers$, extending the applicability to mixed-integer problems.

Optimization problems of type~\labelcref{original} occur in several fields. 
In bilevel optimization, for instance, the objective and constraints of the lower level can be summarized in the optimal-value function.
Typically, there exists no explicit representation of this function, but under certain assumptions
it is globally Lipschitz continuous and the respective Lipschitz constant is computable, \eg, compare~\cite{Gruebel2023} or \cite[Remark 2]{SchmidtSirventWollner2019}. 
Similarly, for optimization problems featuring differential equations as constraints, \eg, as in optimal gas network control~\cite{Koch2015}, 
recent work advertises the use of Neural Networks (NNs) to replace the implicit constraints~\cite{Raissi2019}, rendering it a surrogate-based optimization problem. 
Although providing an explicit constraint instead of an implicit one, a NN typically shows a deeply nested nonlinear structure, which makes a direct inclusion into the problem formulation inadvisable.
The NN may even lack differentiability when so-called ReLU functions are involved.
However, in~\cite{fazlyab-robustnessNNs-2022,pauli-NNLipschitzBounds-2022} the authors present ways to compute global Lipschitz constants for NNs, turning it into a problem of type~\labelcref{original}.
A more detailed account is given below.

In our problem formulation, it is assumed that we can evaluate the constraint function~$\vr$ without having an explicit representation.
Such problems are tackled in the field of \textit{derivative-free optimization}.
Here, the applicability of any particular approach also depends on the structure of~$\Omega$.
For example, if $\Omega$ is defined by linear constraints only, a directional direct-search approach can be leveraged
which applies a Lagrangian reformulation; see~\cite{Lewis1999}.
For the case when non-smooth constraints define $\Omega$, filtering techniques as proposed in~\cite{Audet2004}
may be more suitable. 
Further, in terms of direct search, a well-known method called \textit{mesh adaptive direct search (MADS)}~\cite{Audet2006}
is capable of handling general constraints in $\Omega$, too.
Alternatively, \citeauthor{Berghen2005}~\cite{Berghen2005}
proposes a trust-region interpolation-based approach to tackle such problems. 
For a general overview on constrained derivative-free optimization, we refer to \cite[Chapter~13]{Conn2009}.

However, assuming no further information about the involved functions can be quite limiting. 
For instance, even when minimizing a continuous function over a simple box-constraint domain, in a worst-case situation, a finite bound on the number of iterations to find an optimal solution cannot be expected, see, \eg, \cite[Chapter~6]{Vavasis1991}.
Hence, we assume the availability of more information, in particular, that the function $\vr$ is Lipschitz continuous and that we know a corresponding Lipschitz constant.
Thus, we touch the field of \textit{Lipschitz optimization}.
Traditionally, problems with univariate, Lipschitzian objectives are tackled by saw-tooth cover methods, see~\cite{Piyavskii1972} for the source. 
More general problems, similar to \labelcref{original}, are solved by branch-and-bound schemes or outer approximation, see~\cite[Chapter~XI]{Horst2013} for an overview and below for a more extensive assignment.

\subsection{Contribution}
We propose an outer approximation approach that generalizes classical cutting plane methods. 
Relaxing the constraint function $\vr$, our framework iteratively refines the feasible region by means of norm-induced cuts.
Each cut bounds the distance, \ie, the norm of the difference, of the variable $\vx$ to the solution of the current subproblem.
Such a solution is obtained by oracle calls. 
This already reveals the twofold nature of the presented general problem statement.
On the one hand, there exists no approach in the field of Lipschitz optimization to handle the posed general problem class, to the best of our knowledge. 
The method presented is not even directly restricted to a particular norm, as long as the respective Lipschitz constant is available and the norm is monotonic. 
On the other hand, the resulting subproblems are not explicitly constructed, but the existence of an oracle to solve them is assumed.

Based on this, our contribution includes (i) the introduction of a meta-level method tackling a problem involving Lipschitzian constraints with known Lipschitz constant in the generality of~\labelcref{original}, see~\Cref{sec:algorithm}.
(ii) We investigate the approach in theory in terms of correctness (\Cref{sec:convergence}) as well as under practical assumptions about infeasibility and the problem complexity (\Cref{sec:termination}).
(iii) Additionally, we visualize the functionality of the method, its limits considering local solutions of the subproblems, and the ambiguity resulting from available Lipschitz constants for components of the constraint function, compare~\Cref{sec:examples}. 

In order to account for our assumptions, we discuss the computation of global Lipschitz constants in~\Cref{subsec:computing-lipschitz-constants} and argue for the solvability of the subproblems in~\Cref{subsec:solving-subproblems}.
In~\Cref{sec:conclusion}, we summarize the presented content and point to future research directions.

\subsection{Related Work}
\label{subsec:related-work}

\subsubsection{Lipschitz Optimization}
\label{subsubsec:related-work-Lipschitz-opt}

Lipschitz optimization methodology comes into play when solving optimization problems, for which the involved functions are not necessarily known, but there exist assumptions about their smoothness. 
Specifically, one assumes that these functions are Lipschitz continuous and the Lipschitz constant is known.

It has been a popular field of research, especially in areas with complex dynamics such as chemical plants~\cite{floudas2013deterministic} or for settings in which even evaluating a function requires extensive numerical simulations such as hyperparameter calibration~\cite{ahmed2020combining}.
Most existing work in this domain has focused on Lipschitz continuous objectives but with tractable constraints~\cite{sergeyev2017deterministic, bagirov2014introduction} by leveraging techniques such as domain partitioning~\cite{liuzzi2010partition, paulavivcius2013simplicial} or guided search with sampling~\cite{malherbe2017global}. 

Recent works expand the use of Lipschitz optimization to the setting where there are Lipschitz assumptions on the constraints.
These include the use of penalty-based techniques to convert the constrained optimization problem into an unconstrained one~\cite{di2016direct} and index-based methods which iteratively sample and evaluate new points based on the number of constraints violated at existing points~\cite{strongin2020global}.
Other approaches leverage available (approximate) Lipschitz constants for the constraints by relaxing them and iteratively re-incorporating this information. 
For example, the method proposed in~\cite{SchmidtSirventWollner2019} approximates the original constraint function by a parallelepipedal shape and refines this in every iteration.
Similarly, the approach of \cite{SchmidtSirventWollner2022} entirely neglects the constraint function at first and improves a non-convex outer approximation of the remaining domain step by step.
This is achieved by bounding the distance to the found intermediate points in terms of the infinity norm and re-modeling it in terms of mixed-integer programming techniques.
Those strategies are applied in~\cite{Gruebel2023,Molan2023}

The algorithmic scheme presented in~\cite{SchmidtSirventWollner2022} tackles a simplified version of problem~\labelcref{original} limited to polyhedral constraints on a mixed-integer domain in the subproblem along with one-dimensional implicit Lipschitz constraints.
In this sense, our work can be seen as an extension, respecting vector-valued functions and general norms.
This is discussed in even more detail in~\Cref{rem:distinction-to-martin}.

There also exist papers that are not explicitly categorized as a part of Lipschitz optimization, since they require (sub)gradient information, but that are closely related such that we mention them here briefly.
The extended cutting plane methods for convex and special non-convex problems relax the original nonlinear constraints and iteratively refine the resulting relaxation~\cite{Westerlund1995,Westerlund1998}.
In this sense, our approach can be seen as a generalization of this as well.
A similar strategy is applied by~\cite{Gugat2018}, which can be seen as the origin for the results above~\cite{SchmidtSirventWollner2019,SchmidtSirventWollner2022,Gruebel2023,Molan2023}.
The most recent references~\cite{Gruebel2023,Molan2023} apply tailored algorithms for Lipschitz optimization to Neural Networks constraints, and more generally, to surrogate-based optimization problems.
This gives a glimpse of the broad applicability of this field, which is discussed more extensively in the following.

\subsubsection{Surrogate-based Optimization}
\label{subsubsec:related-work-surrogate-opt}

As stated, our approach is particularly well suited for nonlinear optimization problems which use surrogate models as constraints along with simpler, \eg, convex, domains.
Such a domain can represent an additional source of uncertainty which increases the complexity of the problem. 
For instance, linear constraints can be associated with ellipsoidal uncertainty sets which are equivalent to second-order cone programming (SOCP) problems.
This is particularly relevant for engineering design applications where you need to manage implementation uncertainty along with complex nonlinear response behaviors~\cite{bertsimas2010robust,ozturk2019optimal}.
Examples include aircraft routing problems~\cite{bartholomew2002} and integrated design optimization~\cite{egea2007}, for instance, which are typically handled in derivative-free optimization.
The latter can be tackled by the present method if the solution operator to a differential equation is Lipschitz with a known constant, which is a common assumption. 
We refer the interested reader to~\cite{boukouvala2016} for a review and exemplary problems.

Recently, it has also become popular to incorporate neural networks and other trained machine learning models as constraints in optimization problems~\cite{mistry2021,tsay2021,kronqvist2021}. 
This has led to the development of software to do so efficiently~\cite{ceccon2022,turner2023}. 
In such cases, the complexity of the machine learning model has a significant effect on problem computation.
Our approach can be used as an alternative to address these problems by replacing the machine learning model with iterative cuts, as in many cases it is possible to compute a Lipschitz constant for the machine learning approach, see~\Cref{subsec:computing-lipschitz-constants}.

%% file: algorithm.tex
\section{Norm-induced Cut Method}
\label{sec:algorithm}

In this section, we formulate a method to tackle problem \labelcref{original}. 
From now on, we will omit the subscripts for each norm, where it is unambiguous.  
The method iteratively solves a relaxed problem and adds cuts to enforce feasibility. 
These cuts are determined by a point in $\Omega$, the norm, and the function $\vr$ with its Lipschitz constant~$L$. 
We define them formally below.

\begin{definition}
	\label{def:cut}
	For $\vect{z} \in \reals^m$, let $\vect{z}_+ = (\max(z_1, 0), \dots, \max(z_m, 0))\T$ denote the component-wise maximum of the elements of $\vect{z}$ and 0.
	Now, let $\vy \in \Omega$ such that $\vr(\vy) \nleq \veczero$, \ie, there exists at least one positive component of $\vr$ at $\vy$, and let $L > 0$ define the Lipschitz constant of function $\vr$. 
	Then, for $\vx \in \Omega$, we call the inequality
	\begin{equation}
		\label{eq:def-nic}
		\frac{\norm{\vr(\vy)_+}}{L} \leq \norm{\vx - \vy}, \tag{NC}
	\end{equation}
	a \emph{norm-induced cut in $\vy$}.
\end{definition}
Whenever referring to the \emph{radius of a norm-induced cut}, we consider the constant term $\norm{\vr(\vy)_+}/L$ of~\labelcref{eq:def-nic} for a given $\vy \in \Omega$.
On the one hand, if $\vy$ is feasible to the overall problem~\labelcref{original}, then $\vr(\vy)_+ = \veczero$ and $\norm{\vr(\vy)_+} = 0$.
On the other hand, if $\vy \in \Omega$ is not feasible to problem~\labelcref{original}, but $\vx \in \Omega$ is, 
then~\labelcref{eq:def-nic} describes the minimum distance $\vx$ must have to $\vy$.
Geometrically, this is ensured by removing a norm-induced ``ball'' around $\vy$ from the feasible region~$\Omega$.

This distance or the radius of the ball equals the radius of the norm-induced cut and includes the violation of $\vr(\vx) \leq \veczero$ at the current solution and the Lipschitz constant~$L$. 
Informally, having a greater violation or a smaller $L$ increases the radius of the norm-induced cut and, thus, leads to a larger ``excluded area'' in the next iteration. 

Now, let the feasible set of \labelcref{original} be denoted by
\begin{equation*}
	Q \coloneqq \defset{\vx \in \Omega}{ \vr(\vx) \leq \veczero} \subseteq \Omega.
\end{equation*}
In the following, we define a family of sets that involve norm-induced cuts and which are relaxations to $Q$.

\begin{definition}
	\label{def:sets}
	Let $k \in \naturals$ be a positive integer and $\finseries \subseteq \Omega$ be a sequence such that $\vr(\vx^i) \nleq \veczero$ for all $i = 0, \dots, k-1$. 
	Then, we define 
	\begin{equation*}
		Q_k \coloneqq \Defset{\vx \in \Omega}{\forall i \in \sset{0, \dots, k-1}: \frac{\norm{\vr(\vx^i)_+}}{L} \leq \norm{\vx - \vx^i}},
	\end{equation*}
	the \emph{$k$th relaxed set} of $Q$.
	It is obtained from $\Omega$ by relaxing the inequality constraint $\vr(\vx) \leq \veczero$ and adding norm-induced cuts for each element $\vx^i$ of the sequence for a total of $k$ cuts. 
	For $k = 0$, we naturally define $Q_0 \coloneqq \Omega$. 
\end{definition}
In \Cref{lem:relax}, we show that these $Q_k$ are indeed relaxations of $Q$.
The resulting family of relaxed problems is denoted as
\begin{equation}
	\label{kth-relaxed-prob}
	\tag{$\mathrm{P}_k$}
	\min_{\vx \in Q_k} f(\vx),
\end{equation}
for $k \in \naturals_0$. 
We interpret $k$ as the iteration index of our method.
The following remark shows a connection between two consecutive relaxed problems ($\mathrm{P}_k$) and ($\mathrm{P}_{k+1}$), in particular, between their feasible sets $Q_k$ and $Q_{k+1}$, respectively.

\begin{remark}
	\label{rem:consec_sets}
	Let $k \in \naturals_0$ and let $(\vx^i)_{i=0,\dots,k} \subseteq \Omega$ be the sequence to define the $k$th and the $(k+1)$th relaxed set $Q_k$ and $Q_{k+1}$, respectively. 
	Then, $Q_{k+1}$ is given by
	\begin{equation*}
		Q_{k+1} = Q_k \cap \Defset{\vx \in \Omega}{\frac{\norm{\vr(\incumbent)_+}}{L} \leq \norm{\vx - \incumbent}}.
	\end{equation*}
\end{remark}

In other words, the $(k+1)$th relaxed set results from the $k$th relaxed set by adding the \textit{norm-induced cut} in $\incumbent$, \ie,
\begin{equation*}
	\frac{\norm{\vr(\incumbent)_+}}{L} \leq \norm{\vx - \incumbent}.
\end{equation*}

Now, we are able to state our method for solving problem \labelcref{original}. 
It is called the \emph{\acl{nic} (\labelcref{alg:nic})}, following its main functionality of iteratively adding norm-induced cuts. 

In particular, it starts by relaxing the inequality constraint in \labelcref{original}, $\vr(\vx) \leq \veczero$, and solves the resulting relaxed problem \labelcref{kth-relaxed-prob} for $k = 0$. 
It is assumed that we have a suitable method to solve \labelcref{kth-relaxed-prob} for all $k \in \naturals_0$, which is discussed in more detail in~\Cref{subsec:solving-subproblems}.
If the relaxed problem is infeasible, so is the original one, and an appropriate message is returned. 
Otherwise, the method checks if the solution of the relaxed problem $\vx^{0}$ is also feasible for the original problem by evaluating function $\vr$ at $\vx^0$ and checking for the inequality. 
If it is satisfied, the $\vx^{0}$ is returned as a solution to the original problem. 
Otherwise, a norm-induced cut in the current point $\vx^0$ is added. 
This defines the set~$Q_1$. 
The resulting problem ($\mathrm{P}_1$) is then solved in the next iteration.

In conclusion, the method adds norm-induced cuts as long as the intermediate solutions are infeasible for the original problem \labelcref{original}. 
Below we give a formal description of the \labelcref{alg:nic} method in pseudo code.
We point out that any specific stopping criterion, \eg, approximated satisfaction of constraints, is omitted,
as we examine the method's theoretical behavior for an infinite number of iterations in \Cref{sec:convergence}.
The analysis with regard to practical application, which is satisfied with approximate solutions, is presented in~\Cref{sec:termination}.

\renewcommand{\thealgorithm}{NIC}
\begin{algorithm}
	\floatname{algorithm}{\acf{nic}}\caption{ }
	\label{alg:nic}
	\begin{algorithmic}[1]
		\REQUIRE Set $\Omega$, function $\vr$, Lipschitz constant $L$
		\ENSURE A message of \labelcref{original}'s infeasibility, a solution to \labelcref{original}, or a sequence $\infseries$
		\STATE solve \labelcref{kth-relaxed-prob}\label{step:nic_solve}
		\IF{\labelcref{kth-relaxed-prob} is infeasible}
		\RETURN \lq \labelcref{original} is infeasible\rq \label{step:nic_ret_inf}
		\ELSE
		\STATE denote $\incumbent$ as the solution of \labelcref{kth-relaxed-prob}
		\ENDIF
		\IF{$\vr(\incumbent) \leq \veczero$}
		\RETURN solution $\incumbent$ to \labelcref{original} \label{step:nic_ret_sol}
		\ENDIF
		\STATE set $k \leftarrow k + 1$
		\STATE goto (\labelcref{step:nic_solve})
	\end{algorithmic}
	
\end{algorithm}

As mentioned in the introduction, a reader familiar with the field of Lipschitz optimization might consider our approach as a generalization of the work~\cite{SchmidtSirventWollner2022}.
Although it originated independently, this is a reasonable consideration that we want to discuss in the following. 

\begin{remark}
	\label{rem:distinction-to-martin}
	In~\cite{SchmidtSirventWollner2022}, the authors consider a problem originating from mixed-integer nonlinear optimization. 
	In particular, the problem consists of a linear objective function with linear constraints, as well as implicitly given one-dimensional nonlinear equality constraints.
	Hereby, the optimizing variable $\vx$ has finite bounds and is part of a mixed-integer space, where the nonlinear functions depend on its continuous part only.
	
	In their algorithm, they iteratively add a cut for each function that does not satisfy its equality condition (approximately). 
	The cut is analogous to the norm-induced cut considered here, but uses the infinity norm in the continuous part of the $\vx$-space and a right-hand side of the absolute violation, as well as the global Lipschitz constant.
	Due to the use of the infinity norm, the authors re-model this cut by introducing additional binary variables and linear constraints that uphold the character of a mixed-integer linear problem for the resulting subproblems.
	As a consequence, the dimension of the subproblems grows by a factor of iterations times the number of nonlinear constraints.
	For this algorithm, the authors prove correctness and give a worst-case termination bound for obtaining an approximate solution.
	
	We investigate a more general problem class by considering a compact domain in a vector space over some field (and/or $\integers$, compare \Cref{rem:integers-valid}) instead of the linear constraints in a mixed-integer domain.
	Regarding the nonlinear functions, we go one step further and, for the individual components, assume a lack of information regarding the Lipschitz constant. 
	
	Our method is also based on adding cuts, but with respect to a general monotonous norm.
	Hence, the resulting subproblems inherit an equally general structure involving the set $\Omega$.
	This, however, prevents an explicit re-modeling of the cuts in order to keep the subproblems in their original problem class, but only allows for a discussion in this regard.
	
	In terms of theoretical contribution, our proofs of correctness, at some points, follow a different approach and open a new perspective on the method. 
	In addition, our termination bound is analogous to the results in~\cite{SchmidtSirventWollner2022} yet more general to account for the more general setting.
	We also comment on the availability of a point-wise Lipschitz constant in~\Cref{rem:point-dep-L}, answering a question posed in the conclusion of their paper, 
	and showcase the ambiguity of available component-wise Lipschitz constants on an example (see~\Cref{ex:berthold}).
	
\end{remark}

\subsection{Computing Global Lipschitz Constants}
\label{subsec:computing-lipschitz-constants}

Since our algorithm depends on the availability of a global Lipschitz constant, 
we want to emphasize its computation under different circumstances.

\subsubsection{Continuously Differentiable Constraint Functions}
\label{subsubsec:L-constant-cont-diff}
Although there exist elaborate methods that leverage continuous differentiability to a greater extent, we want to mention it for completeness.
The following is inspired by~\cite[Remark 3.1]{SchmidtSirventWollner2022} yet considers multidimensional functions on real vector spaces and is thus more general.
Since this result is leveraged in our examples in~\Cref{sec:examples}, we provide it in more detail than the remaining parts.

For now, let $\Omega \subseteq \reals^n$ and let $\vect{f}: \Omega \to \reals^m$ be continuously differentiable on an open convex superset $\Gamma$ of $\Omega$.
Let $D\vect{f}$ denote the Jacobian of $\vect{f}$ and consider $\vx,\, \vy \in \Gamma$. 
Then, there exists the following mean-value inequality for vector-valued functions~\cite[Chapter 1.1.2]{Cui2022}.
\begin{equation}
	\label{eq:L-constant-cont-diff}
	\norm[q]{\vect{f}(\vx) - \vect{f}(\vy)} \leq \sup_{t \in [0, 1]} \norm[p, q]{D\vect{f}(\vy + t(\vx - \vy))} \norm[p]{\vx - \vy}.
\end{equation}
Note that here we consider $\reals^n$ and $\reals^m$ to be equipped with a $p$- and $q$-norm, respectively, $p,\, q \in [1, \infty]$. 
The matrix norm is then induced by those vector norms and is defined as 
\begin{equation*}
	\norm[p, q]{A} = \sup\{\norm[q]{A\vx} \mid \vx \in \reals^n,\, \norm[p]{\vx} \leq 1 \},
\end{equation*}
for some $A \in \reals^{m \times n}$.
The proof of~\labelcref{eq:L-constant-cont-diff} can be found in~\cite[\nopp3.2.3]{Ortega2000}, for instance. 

This implies that if we can compute $\sup \norm[{p, q}]{D\vect{f}(\vect{z})}$ for $\vect{z}$ in the convex hull of $\Omega$ or an upper bound for this term,
we  receive a global Lipschitz constant for $\vect{f}$. 
Since $\norm[{p, q}]{\fcdot}$ as well as $D\vect{f}(\cdot)$ are continuous and the compactness of $\Omega$ transfers to its convex hull, at least we know that the supremum exists and takes a finite value.
That is, $\vect{f}$ is globally Lipschitz continuous in this case.

 \subsubsection{Bilevel Optimization}
 \label{subsubsec:L-constant-bilevel}
 
There also exist differentiable constraint functions, where the explicit form and, thus, the Jacobian are unknown. 
In some cases, this occurs in bilevel optimization if one considers the value function $\vr$ of the follower.
A function evaluation of $\vr$ is equivalent to solving (a regularized version of) the lower-level problem, given a fixed upper-level decision. 
The computation of a Lipschitz constant depends on the assumptions for both the upper and the lower level. 

In~\cite[Remark 2]{SchmidtSirventWollner2019}, the authors highlight that under compactness of the feasible set of the upper level, convexity of the lower level problem in the follower variable, and some constraint qualification,
the value function $\vr$ is globally Lipschitz continuous. 
In addition, if uniqueness holds for the solution of the lower level, which is parameterized in the upper level decision, as well as for the dual variables of the lower level, then $\vr$ is differentiable and the Lipschitz constant can be computed using the results in~\cite{gauvin1982}.

This is extended in~\cite{Gruebel2023} to the case of a non-convex quadratic objective function in the lower level. 
Supposing the lower level to be feasible for any upper level decision, the computation of the Lipschitz constant is conducted leveraging a variant of the Hoffman Lemma~\cite{Hoffman1952,Still2018} and linear optimization.

\subsubsection{Neural Networks}

As briefly sketched in the introduction, there exist approaches to approximate dynamical systems by means of neural networks~\cite{Raissi2019}.
This reduces infinite-dimensional problems to finite-dimensional ones, but requires the introduction of deeply nested nonlinear functions, the NNs.
Exactly computing the Lipschitz constants of Neural Networks is known to be an NP-hard problem~\cite{virmaux2018lipschitz, jordan2020exactly}.
As such, most works have focused on developing approximations or algorithms that run in exponential time. The authors of~\cite{virmaux2018lipschitz} provide a method to compute an upper bound to the Lipschitz constants of any automatically differentiable functions by leveraging the computational graph and the individual Lipschitz constants of each operation. In~\cite{latorre2020lipschitz}, the authors take advantage of the fact that for any Lipschitz continuous and differentiable scalar function on a convex domain, the Lipschitz constants can be estimated by maximizing the norm of the gradient on that domain. 
For NNs with differentiable activation functions, they reformulate this gradient maximization problem as a polynomial optimization problem. They then provide several different relaxations of this optimization problem, whose solution provides an upper bound for the Lipschitz constant, which is sufficient for our algorithm. 
A similar methodology for more general activation functions is provided in~\cite{fazlyab2019efficient}, where the authors formulate the Lipschitz estimation problem as a collection of semi-definite programs which span the trade-off between accuracy and scalability. 
Any of these or other existing methods can be used for the estimation of the Lipschitz constants for our algorithm.

\subsubsection{General Functions}
In case no information about the structure of the function is available, we can use sampling-based approaches in order to approximate the Lipschitz constant. 
One such method is outlined in~\cite{wood1996estimation}. 
It estimates the global Lipschitz constant by calculating local Lipschitz constants constructed using random samples of pairs of points on the domain. 
In particular, these local Lipschitz constants are then used to estimate the parameters of a reverse-Weibull distribution, whose location parameter then provides us with an estimate of the global Lipschitz constant.

\subsection{Solving~\labelcref{kth-relaxed-prob}}
\label{subsec:solving-subproblems}

In addition to the availability of a global Lipschitz constant, \labelcref{alg:nic} requires solving the subproblems~\labelcref{kth-relaxed-prob} to global optimality. 
Hereby, the iteration index $k$ equals the number of norm-induced cuts added to the relaxed feasible region~$\Omega$.
According to the properties of a norm, each cut is non-convex.
Although giving the problem an explicit formulation, non-convexities are considered problematic when global solutions are required. 
Hence, we present a discussion about the solvability of the subproblems~\labelcref{kth-relaxed-prob}, leveraging the structural properties resulting from norm-induced cuts. 

The cut structure is closely connected to the specific norm and field involved.
In the current subsection, we consider different $p$-norms, $p \in [1, \infty]$, and $F = \reals$.

Coming to the first edge case $p = \infty$, the resulting cuts are box-shaped and the norm is not differentiable everywhere. 
However, as discussed for instance in~\cite[Lemma 3.3]{SchmidtSirventWollner2022}, one can re-model those types of norm-induced cuts by introducing binary variables and a big-$M$ formulation. 
The latter introduces a potential disadvantage, since it highly depends on the magnitude of the constants $M$.
However, this procedure renders the resulting problem a mixed-integer or mixed-binary linear problem, which is considered to be tractable.
A similar situation applies for the other edge case $p = 1$. 
For completeness, the interested reader can find both re-modeling strategies in~\Cref{sec:reform-cuts}.

For $p \in (1, \infty)$, first note that the corresponding norm is differentiable everywhere except for the origin.
This allows us to compute gradients that can be leveraged in optimization. 
Second, a norm is convex by definition and, thus, a norm-induced cut represents a reverse convex constraint. 
As detailed below, reverse programming incorporates the potential to facilitate the solving process in comparison to general nonlinear constraints. 
Hence, in some sense, using $p$-norm cuts, general nonlinear functions are replaced by constraint types that are still non-convex yet considered more tractable.
For $p = 2$, one can even square both sides of a norm-induced cut and leverage non-convex quadratic programming techniques, which have received great attention in the past.

\Cref{tab:norm_comparison} summarizes the advantages and challenges with respect to the $p$-norms.
As we deem $p \in (1, \infty)$ to be the more challenging part, we discuss solving reverse convex programs in the following.

\begin{table}[h]
    \caption{\label{tab:norm_comparison}Comparison of advantages and challenges for different $p$-norms in norm-induced cuts on $\reals^n$.}
    \begin{tabularx}{\textwidth}{>{\centering\arraybackslash}p{0.9cm}p{2.0cm}XX}
    \toprule
    $p $ & \textbf{type of }\labelcref{kth-relaxed-prob} & \textbf{advantages} & \textbf{challenges} \\ \midrule
    $\{1, \infty\}$ & mixed-integer linear & 
    \begin{enumerate}[label=(\roman*), nosep, left=0pt, labelsep=0.0cm, align=left, before=\leavevmode\vspace{-\baselineskip}] 
    	\item Re-modeling by MIP techniques 
    	\item Solution by standard solvers and established methods
    \end{enumerate} &
    \begin{enumerate}[label=(\roman*), nosep, left=0pt, labelsep=0.0cm, align=left, before=\leavevmode\vspace{-\baselineskip}]
    	\item Introduction of binary component
    	\item Numerical troubles with bad estimations of big-$M$ constant
    \end{enumerate} \\
    $(1, \infty)$ & reverse \mbox{convex} &
    \begin{enumerate}[label=(\roman*), nosep, left=0pt, labelsep=0.0cm, align=left, before=\leavevmode\vspace{-\baselineskip}]
    	\item Differentiable constraint functions
    	\item Reverse convex instead of general nonlinear constraints
    \end{enumerate} & 
    \begin{enumerate}[label=(\roman*), nosep, left=0pt, labelsep=0.0cm, align=left, before=\leavevmode\vspace{-\baselineskip}]
    	\item Subproblem still non-convex
        \item Necessity for elaborate methods from reverse convex programming
    \end{enumerate} \\
    $2$ & quadratic &
    \parbox[t]{\linewidth}{
    \begin{enumerate}[label=(\roman*), nosep, left=0pt, labelsep=0.0cm, align=left, before=\leavevmode\vspace{-\baselineskip}]
    	\item Leveraging quadratic programming advances
    \end{enumerate}
   } & 
    \parbox[t]{\linewidth}{
    \begin{enumerate}[label=(\roman*), nosep, left=0pt, labelsep=0.0cm, align=left, before=\leavevmode\vspace{-\baselineskip}]
    	\item Subproblem still non-convex
    \end{enumerate}
   } \\
    \bottomrule
    \end{tabularx}
\end{table}

First, note that an application including cuts that lower bound a norm by a positive scalar is the obnoxious facility location problem~\cite{Kalczynski2021}.
In particular, those cuts correspond to reverse convex constraints, which appear in several applications~\cite{Jacobsen2009}. 
While problems with reverse convex constraints are difficult to solve, there has been significant work done to develop algorithms to solve them to global optimality. 
The authors of~\cite{tuy1987convex} showed that problems with multiple reverse convex constraints can be converted into a problem with a single reverse convex constraint and a convex constraint, which can then be solved using concave minimization. 
This was followed by several works expanding on this connection and developing algorithms for more general conditions~\cite{Horst2013, thoai1988modified}.
For the specific setting where the reverse convex set is a polyhedron, the feasible region can be written as a union of halfspaces. This observation was leveraged to develop algorithms to solve reverse convex problems with convex sets through inner approximations~\cite{yamada2000inner}. 
A good overview of the variety of methods available is presented in~\cite{horst2013handbook}.
Recent works have leveraged enumerative schemes~\cite{bunin2016extended} and heuristic methods~\cite{drezner2020solving} to solve the reverse convex programs.

We also point to~\cite{hildebrand2024}, which analyses the complexity of the integer feasibility problem of reverse convex sets where the latter share a property called Boundary Hyperplane Cover. 
Assuming $\Omega$ to be polyhedral, we highlight that our subproblems fall into the class of such problems and, thus, according to this article can be solved in polynomial time given the dimension to be fixed.

For the case of 2-norm and continuous variables, the work in~\cite{bienstock2014polynomial} shows that the subproblems are polynomially solvable if the set \(\Omega\) is a polyhedron with a bound on the number of facets and a fixed number of inverse 2-norm constraints. Their proof is constructive and provides an algorithm to solve the problem by enumerating over possible solutions on the faces of the feasible region. 
Another approach is presented by~\cite{beck2017branch} who provide a branch and bound algorithm to tackle the problem.
In case of the presence of integer variables, we can leverage approaches developed to tackle non-convex quadratically constrained quadratic programs such as~\cite{misener2012global}.

\subsection{Definitions}
\label{subsec:definitions}

Now, in order to discuss properties of the \labelcref{alg:nic} method in the following sections, we clarify the meaning of inequalities in a multi-dimensional context and introduce the notation of a ball as well as the terms $\eps$-packing/-covering. 

\begin{remark}
	Given a point $\vx \in \Omega$, it is infeasible for~\labelcref{original} if there exists at least one component $r_i$ of $\vr = (r_1, \dots, r_m)$, $i \in \{1, \dots, m\}$, such that $r_i(\vx) > 0$.
	For the sake of concise presentation, we use $\vr(\vx) \nleq \veczero$ to notate that such a component $r_i$ exists.
\end{remark}

\begin{definition}
	\label{def:ball}
	Let $\eps > 0$, $\bar{\vx} \in X$, and $\norm{\,\cdot\,}$ a norm on $X$. Then, we define the \emph{$\eps$-ball around $\bar{\vx}$} as 
	\begin{equation*}
		B_\eps(\bar{\vx}) \coloneqq \Defset{\vx \in X}{\norm{\vx - \bar{\vx}} < \eps}.
	\end{equation*}
	For the case $\bar{\vx} = \veczero$, we write the short form
	\begin{equation*}
		B_\eps \coloneqq B_\eps(\veczero) = \Defset{\vx \in X}{\norm{\vx} < \eps}.
	\end{equation*}
\end{definition}

For the fundamental mathematics regarding $\eps$-packing/-covering, we refer to \cite{Tikhomirov1993}.
We use the notation from the lecture notes \cite{Han2021} and adapt the definition to fit our question.

\begin{definition}
	\label{def:pack}
	Let $k \in \naturals_0$, $\eps > 0$, and $\norm{\,\cdot\,}$ a norm on $X$.
	Further, consider $A \subseteq X$ and $P \coloneqq {\sset{\vx^0, \dots, \vx^k} \subseteq A}$ a set of points in $A$.
	\begin{itemize}
		\item[a)] If $A \subseteq \bigcup_{i = 0}^k B_\eps(\vx^i)$, we call $P$ an \emph{$\eps$-covering of $A$}. 
		Furthermore, we define the minimal cardinality of such a $P$ as
		\begin{equation*}
			N(A, \norm{\,\cdot\,}, \eps) \coloneqq \min \Defset{k}{\exists\text{$\eps$-covering of $A$ with size }k},
		\end{equation*}
		the \emph{covering number of $A$ \wrt $\eps$}.
		\item[b)] If $\sset{B_{\eps/2}(\vx^i)}_{i = 0,\dots,k}$ are pairwise disjoint, \ie, $\norm{\vx^i - \vx^j} \geq \eps$ for all ${i,\, j = 0, \dots, k}$, $i \neq j$, we call $P$ an \emph{$\eps$-packing of $A$}, 
		since $A$ contains all the centers of the balls.
		We also define the maximal cardinality of such a $P$, \ie, 
		\begin{equation*}
			M(A, \norm{\,\cdot\,}, \eps) \coloneqq \max \Defset{k}{\exists \text{$\eps$-packing of $A$ with size } k},
		\end{equation*}
		the \emph{packing number of $A$ \wrt $\eps$}.
	\end{itemize}
\end{definition}

%% file: convergence.tex
\section{Convergence and Optimality Results}
\label{sec:convergence}

This section is dedicated to showing the correctness of \labelcref{alg:nic}. 
Depending on the feasibility of the original problem \labelcref{original}, we show feasibility and even optimality of a returned solution or of the accumulation points of the sequence of intermediate points produced by the method. 
In the end, we give a statement regarding the availability of local solutions in step~\labelcref{step:nic_solve}.

Throughout the current section, we assume exactness of the underlying methods in order to analyze the theoretical implications.
In practice, however, numerical errors occur, and approximate solutions suffice. 
This is covered by the analysis of the problem complexity in~\Cref{sec:termination}.

We start by remarking on the boundedness of~\labelcref{original} in the case of feasibility. 

\begin{remark}
	As $\vr$ is Lipschitz continuous, it is also continuous. 
	By rewriting the feasible set of \labelcref{original} as
	\begin{equation*}
		Q = \Omega \cap \vr^{-1}((-\infty, 0]^m),
	\end{equation*}
	where $\vr^{-1}(\cdot)$ notes the pre-image of $r$, we can use the closure of $(-\infty, 0]^m$ to derive the closure of $\vr^{-1}((-\infty, 0]^m)$. 
	Together with the compactness of $\Omega$, this implies that  $Q$ is compact. 
	Therefore, problem \labelcref{original} minimizes a continuous function over a compact set and, thus, if $Q \neq \emptyset$, the problem has a finite solution value.
\end{remark}

With this in mind, we state two lemmata. 
The first one ensures that \labelcref{alg:nic} does not find any point twice, 
the second one implies that \labelcref{kth-relaxed-prob} is a relaxation of~\labelcref{original} for each $k \in \naturals$, as mentioned in \Cref{def:sets}. 
In other terms,
we construct an outer approximation of the feasible region.

\begin{lemma}
	\label{lem:cut_off}
	Consider some $k \in \naturals_0$ and a solution $\incumbent$ to \labelcref{kth-relaxed-prob} such that $\vr(\incumbent) \nleq \veczero$. 
	Then, $\incumbent \notin Q_{k+1}$, \ie, the point $\incumbent$ is infeasible for problem $\mathrm{(P}_{k+1}\mathrm{)}$.
\end{lemma}

\begin{proof}
	As $\incumbent$ is a solution to \labelcref{kth-relaxed-prob}, we have $\incumbent \in Q_k$. 
	As $L > 0$ and there exists a strictly positive component of $\vr(\incumbent)$, it follows 
	\begin{equation*}
		\norm{\incumbent - \incumbent} = 0 < \frac{\norm{\vr(\incumbent)_+}}{L}.
	\end{equation*}
	Then, by~\Cref{rem:consec_sets} we derive that $\incumbent \notin Q_{k+1}$ and the claim follows.
\end{proof}

\begin{lemma}
	\label{lem:relax}
	For any $k \in \naturals_0$, it holds that $Q \subseteq Q_k$.
\end{lemma}

\begin{proof}
	In the following, we use the short notation $[m] \coloneqq \{1, \dots, m\}$.
	For a vector $\vect{z} \in \reals^m$ and a subset of indices $I \subseteq [m]$, define $\vect{z}_I \in \reals^m$ such that
	\begin{equation*}
		(\vect{z}_I)_p = \begin{cases}
		z_p,\, p \in I,\\
		0, \, \text{else},
	\end{cases} \text{ for } p \in [m].
	\end{equation*}
	
	Now, if $k = 0$, $Q_k = \Omega$ by construction and the claim follows from \Cref{def:sets}.
	
	Otherwise, let $k \in \naturals$ be arbitrary but fixed and let $\finseries$ be a sequence to determine $Q_k$.
	Then, it holds that
	\begin{equation*}
		\vx^i \in \Omega \qquad \land \qquad \vr(\vx^i) \nleq \veczero,
	\end{equation*}
	for $i = 0,\dots,k-1$, where the second condition is equivalent to $\norm{\vr(\vx^i)_+} > 0$.
	We define the set of non-negative indices of $\vr(\vx^i)$ as $P^i \coloneqq \{p \in [m] \mid r_p(\vx^i) \geq 0\}$.
	Further, consider some $\bar{\vx} \in Q \subseteq \Omega$ and $\vr(\bar{\vx}) \leq \veczero$. 
	We can conclude that 
	\begin{equation*}
		0 < \cutlhs{i} \leq \frac1L \norm{(\vr(\bar{\vx}))_{P^i} - \vr(\vx^i)_+} \leq \frac{1}{L}\norm{\vr(\bar{\vx}) - \vr(\vx^i)} \leq\norm{\bar{\vx} - \vx^i},
	\end{equation*}
	for all $i = 1,\dots,k-1$, where the second and third inequalities follow from the assumed monotonicity of the norm.
	In particular, the components of the argument in the norm are non-decreasing in absolute terms from left to right, and therefore the norm expressions are non-decreasing as well. 
	Hence, for all $i = 0,\dots,k-1$, $\bar{x}$ satisfies the norm-induced cut in $\vx^i$. 
	With \Cref{def:sets} it follows $\bar{\vx} \in Q_k$ and the claim is proven.
\end{proof}

\Cref{lem:relax} shows that infeasibility of problem \labelcref{kth-relaxed-prob} implies the infeasibility of problem \labelcref{original}, 
in particular, $Q_k = \emptyset \Rightarrow Q = \emptyset$. 
Therefore, step~\labelcref{step:nic_ret_inf} associated with the corresponding if-clause is reasonable 
and \labelcref{alg:nic} gives a correct answer if it stops in step~\labelcref{step:nic_ret_inf}. 
Furthermore, it also follows from \Cref{lem:relax} that \labelcref{kth-relaxed-prob} is a relaxation of \labelcref{original} for $k \in \naturals_0$.
Therefore, the solution value of \labelcref{kth-relaxed-prob}, which is obtained during \labelcref{alg:nic}, gives a lower bound to the value of \labelcref{original}.

In the following, we prove that \labelcref{alg:nic} either converges to a feasible solution in finite time or produces a sequence with feasible accumulation points for \labelcref{original}.

\begin{theorem}\label{thm:feas}
	Let \labelcref{original} be feasible, \ie, $Q \neq \emptyset$. Then, either
	\begin{itemize}
		\item[a)] \labelcref{alg:nic} stops with a feasible solution $\sol$ for \labelcref{original} in step \labelcref{step:nic_ret_sol}, or
		\item[b)] \labelcref{alg:nic} creates a sequence of points $\infseries$, which has at least one convergent subsequence,
		each of which has a limit feasible for \labelcref{original}.
		That is, $\infseries$ has at least one accumulation point and all accumulation points of $\infseries$ are feasible for \labelcref{original}.
	\end{itemize}
\end{theorem}

\begin{proof}
	With the assumption and \Cref{lem:relax}, it is clear that
	\begin{equation}
		\label{eq:subset_nonempty}
		\emptyset \neq Q \subseteq \bigcap_{k} Q_k.
	\end{equation}
	Here, $k$ indicates the iteration index and is specified in the case distinction below.
	From~\labelcref{eq:subset_nonempty} we have that $Q_k \neq \emptyset$ and, thus, \labelcref{alg:nic} does not stop in step~\labelcref{step:nic_ret_inf}. 
	Hence, it either stops in step~\labelcref{step:nic_ret_sol} or runs for infinite time. 
	We start with the former case.
	
	a) Let $k \in \naturals_0$ be the index of the stopping iteration. 
	According to \Cref{def:sets}, $Q_k \subseteq \Omega$ and, thus, the current solution satisfies $\incumbent \in \Omega$.
	Since stopping in step~\labelcref{step:nic_ret_sol} requires $r(\incumbent) \leq \veczero$, 
	it follows that \labelcref{alg:nic} returns $\sol \coloneqq \incumbent \in Q$, \ie, a feasible point for \labelcref{original}.
	
	b) \labelcref{alg:nic} runs for infinite time, \ie, stopping in step~\labelcref{step:nic_ret_inf} is not possible and we assume that we never satisfy the criterion in step~\labelcref{step:nic_ret_sol}.
	Again, as $Q_k \subseteq \Omega$ for all $k \in \naturals_0$, we receive a sequence $\infseries \subseteq \Omega$. 
	With $\Omega$ being compact, $\infseries$ has at least one convergent subsequence $\infsubseries$ with 
	\begin{equation*}
		\lim_{j \rightarrow \infty} \subincumbent = \sol \in \Omega.
	\end{equation*}
	We prove the feasibility of $\sol$ for \labelcref{original} by contradiction and assume $\sol \notin Q$, \ie, $\vr(\sol) \nleq \veczero$. 
	In particular, there exists $\eps \coloneqq \norm{\vr(\sol)_+} > 0$.
	
	As $\vr$ is Lipschitz continuous, it is also continuous. 
	With the continuity of $\norm{\fcdot}$ and the component-wise maximum of the argument and zero, we can conclude the continuity of the entire expression $\norm{\vr(\fcdot)_+}$.
	Hence, there exists $j' \in \NN$ such that $\norm{\vr(\subincumbent)_+} > \eps/2$ for all $j \geq j'$.
	Further, as the sequence $\infsubseries$ converges to $\sol$, there exists $j'' \in \NN$ such that $\norm{\vx^{i_{j_1}} - \vx^{i_{j_2}}} \leq \eps/(2L)$ for all $j_1, \, j_2 \geq j''$.
	
	Now, let $j_1,\, j_2 \geq \max(j', j'')$ with $j_1 < j_2$. 
	According to the above considerations, we can derive
	\begin{equation}
		\label{ieq:feas-contradiction}
		\frac{\eps}{2L} < \frac{\norm{\vr(\vx^{i_{j_1}})_+}}{L} \leq \norm{\vx^{i_{j_2}} - \vx^{i_{j_1}}} \leq \frac{\eps}{2L},
	\end{equation}
	which is a contradiction.
	The first non-strict inequality follows from the fact that $\vx^{i_l} \in Q_{i_k}$, compare~\labelcref{eq:def-nic}.

\end{proof}

In other terms, this shows that \labelcref{alg:nic} solves the feasibility problem of \labelcref{original} in case it really is feasible. 
An alternative proof of part b) for $\RR$ (and fields isomorphic to it) includes a covering argument, but is not valid for the generality we assume here.
However, this alternative comes in handy when one investigates the running time of~\labelcref{alg:nic} for infeasible problems. 
Therefore, we present it in a brief version in~\Cref{subsec:termination-infeasible} and give exemplary running times there.

The following remark takes the availability of point-dependent, thus smaller, Lipschitz constants into account and comments on their effects on \Cref{thm:feas}.

\begin{remark}
	\label{rem:point-dep-L}
	Let's assume we have access to the point-dependent Lipschitz constants $L_\vx$ of $\vr$ for each $\vx \in \Omega$. That is,
	\begin{equation}
		\label{eq:local_lip}
		\forall \vx \in \Omega\, \exists L_\vx > 0: \norm{\vr(\vx) - \vr(\vy)} \leq L_\vx \norm{\vx - \vy}, \qquad \text{for all } \vy \in \Omega.
	\end{equation}
	First, note that the statements of \Cref{lem:cut_off} and \Cref{lem:relax} still hold when using the point-dependent Lipschitz constants $L_{\incumbent}$ and $L_{\vx^i}$ in the proofs, respectively. 
	Second, the global Lipschitz constant $L$ is an upper bound to every point-dependent one, \ie, $L \geq L_\vx$ for all $\vx \in \Omega$.
	Therefore, we can use an upper bound of $\eps/(2L_{\vx^{i_{j_1}}})$ before the first strict inequality in~\labelcref{ieq:feas-contradiction} leading to the same result.
	Further, the availability of point-dependent Lipschitz constants can lead to stronger cuts, as the radius of~\labelcref{eq:def-nic} potentially increases, and may result in faster convergence.
	
	Further, during the execution of \labelcref{alg:nic} in iteration $k \in \naturals_0$, we have to consider the Lipschitz continuity of $\vr$ only on $Q_k$. 
	We note that, while this has no influence on the behavior of the algorithm per se, it might be possible to leverage a bounded domain such as $Q_k$ to compute the Lipschitz constant efficiently, as well as obtain a lower constant.
	
	It is also possible to combine pointwise dependence with the current feasible region.
	In particular, if we are in iteration $k \in \naturals_0$ and have solved \labelcref{kth-relaxed-prob} to receive $\incumbent$, 
	we could compute $L_{\incumbent}$ for $\vy \in Q_k$ instead of doing so for all $\vy \in \Omega$ as denoted in~\labelcref{eq:local_lip}.
	Then, we could proceed and add the cut with respect to $L_{\incumbent}$.
\end{remark}

Besides the question of the feasibility of the output, its quality/optimality also needs to be investigated.
Note that \Cref{thm:feas} does not require the solution in step~\labelcref{step:nic_solve} to be optimal per se. 
Therefore, \labelcref{alg:nic} produces a feasible output even if \labelcref{kth-relaxed-prob} is only solved for a feasible solution. 
However, when investigating the optimality of the returned solution or the accumulation points,
we assume the global optimality of $\incumbent$ for \labelcref{kth-relaxed-prob} in step~\labelcref{step:nic_solve} 
and apply a case distinction as in \Cref{thm:feas}.

\begin{theorem}\label{thm:opt}
	Let \labelcref{original} be feasible, \ie, $Q \neq \emptyset$, 
	and let step~\labelcref{step:nic_solve} produce a globally optimal point $\incumbent$ for \labelcref{kth-relaxed-prob}, $k \in \naturals_0$. 
	Then, either
	\begin{itemize}
		\item[a)] \labelcref{alg:nic}~stops with a globally optimal solution $\sol$ of \labelcref{original} in step~\labelcref{step:nic_ret_sol}, or
		\item[b)] \labelcref{alg:nic}~creates a sequence of points $\infseries$, which has at least one convergent subsequence, 
		each of which has a limit that is globally optimal for \labelcref{original}.
		That is, $\infseries$ has at least one accumulation point and all accumulation points of $\infseries$ are globally optimal for \labelcref{original}.
	\end{itemize}
\end{theorem}

\begin{proof}
	With \labelcref{original} being feasible, \Cref{thm:feas} shows that \labelcref{alg:nic} either stops in step~\labelcref{step:nic_ret_sol} or iterates infinitely often. 
	We start with the former case.
	
	a) Let $k\in\naturals_0$ be the index of the stopping iteration, in which $\incumbent$ is returned. 
	According to \Cref{thm:feas}, $\incumbent$ is feasible for \labelcref{original}, \ie, $\incumbent \in Q$, and, thus, it is
	\begin{equation*}
		\min_{\vx \in Q} f(\vx) \leq f(\incumbent).
	\end{equation*}
	Further, by assumption $\incumbent$ is optimal for~\labelcref{kth-relaxed-prob} and \Cref{lem:relax} states $Q \subseteq Q_k$ which implies 
	\begin{equation*}
		f(\incumbent) = \min_{\vx \in Q_k} f(\vx) \leq \min_{\vx \in Q}f(\vx).
	\end{equation*}
	Together, we conclude 
	\begin{equation*}
		\min_{\vx \in Q} f(\vx) = f(\incumbent),
	\end{equation*}
	which proves claim a).
	
	b) Assuming \labelcref{alg:nic} iterates infinitely, we denote $\infseries$ as the resulting sequence of intermediate solutions to the respective subproblems.
	From \Cref{thm:feas}, we derive that the series has at least one accumulation point and that all its accumulation points are feasible for \labelcref{original}. 
	Hence, let $\infsubseries$ denote a convergent subsequence of $\infseries$ with limit $\sol \in Q$. 
	The feasibility of $\sol$ for \labelcref{original} immediately gives 
	\begin{equation*}
		\min_{\vx \in Q} f(\vx) \leq f(\sol).
	\end{equation*} 
	Applying \Cref{lem:relax} on the index of the subsequence, leads to $Q \subseteq Q_{i_j}$ and, thus, 
	\begin{equation*}
		f(\subincumbent) = \min_{\vx \in Q_{i_j}} f(\vx) \leq \min_{\vx \in Q}f(\vx),
	\end{equation*}
	for $j \in \naturals_0$. 
	Leveraging the continuity of $f$ when taking the limit, we can derive
	\begin{equation*}
		\lim\limits_{j \rightarrow \infty} f(\subincumbent) = f(\lim\limits_{j \rightarrow \infty} \subincumbent) = f(\sol) \leq \min_{\vx \in Q} f(\vx).
	\end{equation*}
	Therefore, $f(\sol) = \min_{\vx \in Q} f(\vx)$ which concludes the proof.
\end{proof}

So far, we have proven that \labelcref{alg:nic} is suited to tackle feasible problems which have a structure like~\labelcref{original}.
We can further extend the method when either Lipschitz constants are available for components of the constraint function $\vr$ or 
the function in question is defined on a lower-dimensional subspace of the feasible set.
The combination of these two extensions is also treatable.
A formalization of this is given below.

\begin{remark}
	\label{rem:separate-cuts}
	Let $\vr = (\vr_p)_{p =1,\dots,R}$ for $R \in \naturals$, \ie, the function $\vr$ constitutes of (possible multi-dimensional) components $\vr_p$. 
	Here, we assume that Lipschitz constants $L_p > 0$ are known and the functions $\vr_p : \Omega \rightarrow \reals^{m_p}$, $m_p \in \naturals$, can be evaluated. 
	Then, we can re-write \labelcref{original} as 
	\begin{equation}
		\label{several_ineq}
		\begin{aligned}
			\min_{\vx}\ & f(\vx) & \\
			\text{s.t.}\ & \vr_p(\vx) \leq \veczero, & \text{for } p = 1,\dots,R,\\
			& \vx \in \Omega.
		\end{aligned}\tag{$\mathrm{P}_R$}
	\end{equation}
	The question arises whether and how the performance of~\labelcref{alg:nic} can profit from this situation. 
	
	First, the method requires adjustments.
	Given the solution $\incumbent$ in iteration $k$, a check for each $\vr_p(\incumbent) \leq \veczero$ before step~\labelcref{step:nic_ret_sol} is required.
	Let $\Pi\define \{p \mid \vr_p(\incumbent) \nleq \veczero\}$.
	If $\Pi \neq \emptyset$, an analogous procedure could add a norm-induced cut for every $p \in \Pi$. 
	
	However, note that all such norm-induced cuts bound $\norm{\vx - \incumbent}$ from below by their radius. 
	This allows to define the maximal radius in iteration $k$ as $C_k \define \max_{p \in \Pi} \norm{\vr_p(\incumbent)_+}/L_p$ and to only add the cut
	\begin{equation*}
		\norm{\vx - \incumbent} \geq C_k.
	\end{equation*}
	One can see that the proofs above hold equally for such a procedure.
	
	Now, we aim to compare the magnitude of $C_k$ with the radius of~\labelcref{eq:def-nic} in~$\incumbent$, say $C \define \norm{\vr(\incumbent)_+}/L$.
	For sake of clarity, consider $\vr = (r_1, r_2)$ and let $\norm[i]{\fcdot}$ be the lower dimensional equivalent of $\norm{\fcdot}$, $i =1,\, 2$.
	Then, on the one hand, the assumed monotonicity of the norm gives
	\begin{equation*}
		\norm[1]{r_1(\vx) - r_1(\incumbent)} = \norm{(r_1(\vx), \veczero) - (r_1(\incumbent), \veczero)} \leq \norm{\vr(\vx) - \vr(\incumbent)} \leq L \norm{\vx - \incumbent},
	\end{equation*}
	and an analogous result for $r_2$.
	In particular,  the Lipschitz constant $L$ of $\vr$ is a Lipschitz constant for $r_1,\, r_2$ such that without loss of generality $L \geq L_1$, $L \geq L_2$. 
	On the other hand, using the monotonicity once more, it holds that
	\begin{equation*}
		\norm[1]{r_1(\incumbent)_+} = \norm{(r_1(\incumbent), \veczero)_+} \leq \norm{\vr(\incumbent)_+},
	\end{equation*}
	and analogously for $r_2$.
	
	In conclusion, it is unclear whether $C_k$ exceeds $C$ or vice versa. 
	This depends heavily on the difference between the known Lipschitz constants $L_p$ and $L$ as well as on the violation caused by the current point.
	This ambiguity is demonstrated in a specific instance in~\Cref{ex:berthold}.

\end{remark}
	
\begin{remark}
	If $\vr$ in \labelcref{original} maps from a lower dimensional subset of $\Omega$ into the $\reals^m$, 
	we can re-write $\vx = (\vx', \vx'') \in \Omega = \Omega' \times \Omega''$ with $\dim(\Omega'),\, \dim(\Omega'') \geq 1$, $\dim(\Omega') + \dim(\Omega'') = n$, and $\vr: \Omega' \rightarrow \reals^m$.
	Denoting the first $\dim(\Omega')$ elements of a vector with the prime, we can investigate the case of a point $\incumbent$ with $\vr((\incumbent)') \nleq \veczero$.
	As a vector $((\incumbent)', \vx'')$, for all $\vx'' \in \Omega''$, can never be a valid solution for \labelcref{original}, the norm-induced cut in $\incumbent$ can be added in $(\incumbent)'$ instead. 
	That is, we can add
	\begin{equation*}
		\norm{\vx' - (\incumbent)'} \geq \norm{\vr((\incumbent)')_+}/L,
	\end{equation*}
	with an appropriate norm for $\Omega'$.
	This may lead to faster convergence, as cuts are added in a potentially smaller subspace $\Omega'$.
\end{remark}

We showed that \labelcref{alg:nic} works correctly when we have a global solution method for~\labelcref{kth-relaxed-prob}. 
But the question arises whether the algorithm still fulfills the desired behavior when only a local solution method is available. 
In the following, we verify that if only local solutions can be found in step~\labelcref{step:nic_solve}, a point returned in step~\labelcref{step:nic_ret_sol} is indeed a locally optimal point.

\begin{theorem}
	\label{thm:local}
	Let \labelcref{original} be feasible, \ie, $Q \neq \emptyset$, and let step~\labelcref{step:nic_solve} produce a locally optimal point $\incumbent$ for \labelcref{kth-relaxed-prob}, $k \in \naturals_0$. 
	Then, if \labelcref{alg:nic} stops in step~\labelcref{step:nic_ret_sol}, it returns a locally optimal solution $\sol$. 
\end{theorem}

\begin{proof}
	Assume that \labelcref{alg:nic} stops in step~\labelcref{step:nic_ret_sol} in iteration $k \in \naturals_0$. 
	As $\incumbent$ is a locally optimal solution for \labelcref{kth-relaxed-prob}, there exists an $\eps > 0$ such that
	\begin{equation*}
		f(\incumbent) = \min_{\vx \in B_\eps(\incumbent) \cap Q_k} f(\vx).
	\end{equation*}
	By \Cref{thm:feas} $\incumbent \in Q$ and, thus, $\incumbent \in Q\cap B_\eps(\incumbent)$. 
	We infer 
	\begin{equation*}
		\min_{\vx \in Q \cap B_\eps(\incumbent)} f(\vx) \leq f(\incumbent).
	\end{equation*}
	Further, the result from \Cref{lem:relax} implies $Q \cap B_\eps(\incumbent) \subseteq Q_k \cap B_\eps(\incumbent)$. 
	We can derive
	\begin{equation*}
		f(\incumbent) = \min_{\vx \in Q_k\cap B_\eps(\incumbent)} f(\vx) \leq \min_{\vx \in Q \cap B_\eps(\incumbent)} f(\vx),
	\end{equation*}
	and, thus, conclude that 
	\begin{equation*}
		\min_{\vx \in Q \cap B_\eps(\incumbent)} f(\vx) = f(\incumbent).
	\end{equation*}
	Therefore, $\incumbent$ is a local optimum for \labelcref{original} and the proof is complete.
\end{proof}

If \labelcref{alg:nic} does not terminate in step~\labelcref{step:nic_ret_sol} while returning a point, we can, in general, not guarantee the local optimality of the accumulation points. On the contrary, the sequence might even converge to the ``worst possible'' point, as we show in \Cref{ex:local}.
We conclude the section with a remark regarding alternatives for the field $F$.

\begin{remark}
	\label{rem:integers-valid}
	Note that replacing $F$ with the ring $\integers$ does not change the validity of the computations above. 
	That is, one can consider $\Omega$ as a subset of a mixed-integer domain as well.
\end{remark}

%% file: termination.tex
\section{Termination and Problem Complexity}
\label{sec:termination}

This section discusses a termination criterion for \labelcref{alg:nic} based on the properties of~$\vr$ in case of infeasibility in \labelcref{original} and the problem complexity.
We restrict ourselves to $F = \reals$ if not stated otherwise.

\subsection{Termination for Infeasible Problems}
\label{subsec:termination-infeasible}

Let \labelcref{original} be infeasible, \ie, $Q = \emptyset$. 
Hence, $\vr(\vx) \nleq \veczero$ and $\norm{\vr(\vx)_+} > 0$ for all $\vx \in \Omega$. 
Due to the compactness of $\Omega$ and the continuity of $\norm{\vr(\fcdot)_+}$, as noted in the proof of~\Cref{thm:feas}, there exists $\delta \coloneqq \min_{\vx \in \Omega} \norm{\vr(\vx)_+} >0$. 

Let $\infseries$ denote a sequence defined by~\labelcref{alg:nic}.
As any new solution $\vx^i$ must satisfy~\labelcref{eq:def-nic} with respect to all previous solutions $\vx^0, \dots, \vx^{i-1}$ (compare \Cref{rem:consec_sets}), 
we can derive for all ${l \in \naturals}$ that 
\begin{equation*}
	\norm{\vx^l - \vx^i} > \frac{\delta}{L}, \qquad \text{for all } i = 0, \dots, l - 1.
\end{equation*} 
Thus, the number of iterations until termination can be bounded by the maximal number of points in $\Omega$ which have a pairwise distance of more than $\delta/L$ with respect to the norm.
This refers to a maximal $\delta/L$-packing of $\Omega$ and was formalized in \Cref{def:pack}.

In particular, our aim is to compute $\packnum$
or an upper bound for it. 
To the best of our knowledge, there is no such bound for a general normed finite-dimensional vector space. 
Though when restricting $X$ to $\reals^n$, the situation differs.

\begin{lemma}[\cite{Tikhomirov1993, Han2021}]
	\label{rem:termination_bound}
	Let $X = \reals^n$ and $\vol(\,\cdot\,)$ denote the Lebesgue measure (``volume''). 
	Further, we define $B \coloneqq B_1$ as the unit ball under the considered norm~$\norm{\,\cdot\,}$ and set ${B/2 \coloneqq B_{1/2}}$. 
	Then, it holds 
	\begin{equation}
		\label{ieq:volume}
		\packnum \leq T,
	\end{equation}
	where
	\begin{equation*}
		T \coloneqq \frac{\vol(\Omega + (\delta/L)B/2)}{\vol((\delta/L)B/2)},
	\end{equation*}
	and the sum of two sets is considered to be the Minkowski sum.
\end{lemma}

Depending on the structure of $\Omega$ and the specific choice of norm, this leads to upper bounds dependent on $\delta$ and $L$. 
We have computed some examples in~\Cref{sec:volume_ex} and give an overview of the results in \Cref{tab:iter_bounds}. 

\begin{remark}
	We note that one can follow a similar line of argumentation even for arbitrary fields by leveraging an interesting relation between the packing and covering numbers $M(\eps)$ and $N(\eps)$.
	In particular, for normed spaces, it is well known that $M(2\eps) \leq N(\eps)$ for any $\eps > 0$ (see, \eg,~\cite{Tikhomirov1993}) and, thus, $M(\delta/L) \leq N(\delta/(2L))$. 
	Then, finite termination per se is already guaranteed if there exists a finite $\delta/(2L)$-covering, and the number of covers $N(\delta/(2L))$ serves as an upper bound for the running time.
	
	For a finite field, this is directly bounded by the number of elements, as none of them is visited twice. 
	For infinite ones with an isomorphism $\phi$ to $\QQ$ (countable) or $\RR$ (uncountable), termination itself follows an application of the theorem of Heine-Borel (see, \eg, \cite[Theorem~2.41]{Rudin1976}). 
	This claims the existence of a finite cover in $\RR^n$ and, vice versa, in the vector space induced by the field.
	A finite bound for any other field is dependent on the field's structure and -- to the best of our knowledge -- requires a case-by-case analysis.	
\end{remark}

\begin{table}
	\caption{Examples for $T$ in \labelcref{ieq:volume} for $\delta/L < 1$ depending on the structure of $\Omega$ and the choice of norm.}
	\label{tab:iter_bounds}
	\centering
	\begin{tabular}{llc}
		\toprule
		$\Omega$ & $\norm{\,\cdot\,}$ & $T$ \\
		\midrule
		$\Omega_1 \coloneqq \prod_{j = 1}^n[a_j,\, b_j]$ & $\norm[\infty]{\,\cdot\,}$ & $\prod_{j = 1}^n \left[(L/\delta)(b_j - a_j) + 1\right]$ \\
		$\Omega_2 \coloneqq \prod_{j = 1}^n[a_j, \, b_j] \cap \integers^n$ & $\norm[\infty]{\,\cdot\,}$ & $\abs{\Omega_2}$ \\
		$\Omega_3 \coloneqq \Defset{\vx}{\norm[2]{\vx} < D}$ & $\norm[2]{\,\cdot\,}$ & $\left(2LD/\delta + 1\right)^n$\\
		\bottomrule
	\end{tabular}
\end{table}

\subsection{Problem Complexity}

It is inadequate to expect to find exact solutions for abstract problem classes, as we consider them, in finite time; see \cite{Nemirovsky1983,Vavasis1991}.
Hence, for the investigation of problem complexity, we only consider approximate solutions.
\begin{definition}
	\label{def:approx_sol}
	Let $\eps > 0$ and let $r_q$ denote the $q$-th component of $\vr$, $q = 1,\dots,m$.
	Then, we call $\sol \in \Omega$ an \emph{$\eps$-approximate solution} for \labelcref{original} if and only if
	\begin{equation}
		\label{eq:approx_subcomponents}
		\forall q = 1,\dots,m \; : \; r_q(\sol) \leq \varepsilon.
	\end{equation}
\end{definition}

As shown in~\cite{Vavasis1991}, continuity of $f$ is not sufficient to find an approximate solution in a finite number of steps.
Thus, we assume the objective $f$ to be also Lipschitz. 
Additionally, we assume the interior of $\Omega$ to be non-empty, 
and that there exists an $\eps > 0$ such that an optimal solution $\sol$ to \labelcref{original} is contained in an $\eps$-ball inside~$\Omega$.

This ensures that if \labelcref{alg:nic} adds norm-induced cuts with a radius of $\eps$,
it is not possible to cut off every neighboring point in $\Omega$.

For the lower bound on the problem complexity, we use a result from~\cite{Nemirovsky1983}. 
We define the following constants regarding $\Omega$.

\begin{definition}[\cite{Nemirovsky1983}]
	\label{def:rad_asph}
	Let $\eps > 0$ and let $\Omega$ have a non-empty interior.
	Then, we set the \emph{radius of $\Omega$} as 
	\begin{equation*}
		\rho(\Omega) \coloneqq \min\defset{s}{\exists \vx \in \Omega : \norm{\vx - \vy} \leq s, \text{ for all } \vy \in \Omega}.
	\end{equation*}
	Further, the \emph{asphericity of $\Omega$} is defined as 
	\begin{equation*}
		\alpha(\Omega) = \inf \defset{\beta}{\exists \vx,\, \vy \in \Omega,\, s \geq 0 : B_s(\vx) \subseteq \Omega \subseteq B_{\beta s}(\vy)}.
	\end{equation*}
	In other terms, the asphericity defines the \emph{ratio of radii of the minimal ball containing $\Omega$ and the maximal ball contained in $\Omega$}.
\end{definition}

Since the magnitude of each $r_q$ has a great influence on how to choose a proper approximation guarantee $\eps$, 
we use the radius $\rho(\Omega)$ and assume $\vr$ to be divided by the constant $\rho(\Omega) L$, ``normalizing'' the inequality, \ie, $\vr \leftarrow \vr/(\rho(\Omega)L)$.

Now, when dealing with problem complexity, we consider the number of oracle calls (step~\labelcref{step:nic_solve}) until achieving an approximate solution.
We denote this number by $T(\eps)$ for $\eps > 0$.
So far, the \labelcref{alg:nic} method has been developed to find exact solutions.
As such, we modify it as follows:
\begin{itemize} 
	\item[a)] We only check for the component-wise inequalities~\labelcref{eq:approx_subcomponents} in the if-clause before step~\labelcref{step:nic_ret_sol},
	\item[b)] and we add the norm-induced cut from \Cref{def:cut} with respect to a slightly changed radius of $\max(\norm{\vr(\incumbent)_+}/L,\, \eps)$.
\end{itemize}

Given these modifications, the number of oracle calls required to find the $\eps$-approximate solution is given by the following theorem. 
\begin{theorem}[\cite{Nemirovsky1983}]
	Let $\eps > 0$ and let $\Omega$ have a non-empty interior. Then,
	\begin{equation*}
		\label{ieq:complex_lb}
		T(\eps) \geq \left(\frac{c}{\alpha(\Omega)}\right)^n \frac{1}{\eps^n},
	\end{equation*}
	with some constant $c > 0$.
\end{theorem}

For the upper bound, we use \Cref{rem:termination_bound} with $\eps$ instead of $\delta/L$ .
That is, instead of giving a termination bound dependent on the values of $\vr$, 
we do it in terms of the approximation guarantee $\eps$.
We formalize this in the following.

\begin{theorem}
	Let $\eps > 0$ and let $\Omega$ have a non-empty interior. Then,
	\begin{equation*}
		\label{ieq:complex_ub}
		T(\eps) \leq (2\rho(\Omega) + \eps)^n \frac{1}{\eps^n}.
	\end{equation*}
\end{theorem}
\begin{proof}
	With the definition of the radius in \Cref{def:rad_asph}, there exists $\bar{\vx} \in \Omega$ such that $\Omega \subseteq B_{\rho(\Omega)}(\bar{\vx})$.
	Together with \Cref{rem:termination_bound}, we conclude
	\begin{equation*}
		\begin{aligned}
			T(\eps) & \leq \frac{\vol(\Omega + \eps B/2)}{\vol(\eps B/2)} \leq \frac{\vol(B_{\rho(\Omega)} + \eps B/2)}{\vol(\eps B/2)} \\
			& = \frac{\vol(B_{\rho(\Omega) + \eps/2})}{\vol(\eps B/2)} = \left(\frac{\rho(\Omega) + \eps/2}{\eps/2}\right)^n \\
			& = \left(\frac{2\rho(\Omega)}{\eps} + 1\right)^n = (2\rho(\Omega) + \eps)^n \frac{1}{\eps^n}.
		\end{aligned}
	\end{equation*}
\end{proof}

In conclusion, considering the dimension $n$ fixed, we observe that the lower and upper complexity bounds show identical order of magnitude in $\eps$. 

%% file: examples.tex
\section{Computational Examples and Illustrations}
\label{sec:examples}

Having investigated the problem complexity, we now present three examples with distinct purposes.
For an overview, we summarize them in the following.
\begin{description}
	\item[\Cref{ex:nic}] We demonstrate the behavior of~\labelcref{alg:nic} on a two-dimensional problem that allows us to illustrate the process.
	\item[\Cref{ex:local}] We show that dropping the global optimality assumption from~\Cref{thm:opt} on step~\ref{step:nic_solve} of~\labelcref{alg:nic} can lead to convergence to a feasible but ``worst possible'' point.
	\item[\Cref{ex:berthold}] We draw a comparison between the availability of tight Lipschitz constants for components of $\vr$  as well as for the entire function $\vr$ as discussed in theory in~\Cref{rem:separate-cuts} on a problem from the literature.
\end{description}

\begin{example}[Illustration of~\labelcref{alg:nic}]
	\label{ex:nic}
	
	Consider the nonlinear problem
	\begin{equation}
		\label{example_nic}
		\begin{aligned}
			\min_{x_1,\, x_2}\ & \abs{x_1 - x_2} + x_1 \\
			\text{s.t.}\ & -\sin(x_1) - x_2 \leq 0,\\
			& x_1, \, x_2 \in [-1, 1].
		\end{aligned}
		\tag{$\mathrm{P}_\mathrm{sin}$}
	\end{equation}
	
	Following our notation, we have $f(\vx) = f(x_1, x_2) = \abs{x_1 - x_2} + x_1$, ${r(\vx) = -\sin(x_1) -x_2}$, and $\Omega = [-1, 1]^2$. 
	The optimal solution of problem~\labelcref{example_nic} is $\sol = (0, 0)$. 
	See \Cref{subfig:ex_nic_feas} for a graphical representation of $\sol$ and the feasible region $Q = \Omega \cap \Defset{\vx \in \reals^2}{r(\vx) \leq 0}$.
	In the following, we omit the bold format and the subscript `+', as $r$ is 1-dimensional.
	
	We specify an accuracy of $\eps = 10^{-4}$.
	Further, we consider the $\reals^2$ with the Euclidean norm $\norm[2]{\,\cdot\,} = \sqrt{x_1^2 + x_2^2}$ and the $\reals$ with the absolute value $\abs{\,\cdot\,}$.
	Computing the Lipschitz constant by following~\Cref{subsubsec:L-constant-cont-diff}, we receive $L = \sqrt{2}$.
		
	In the first iteration $k \gets 0$, the algorithm solves ${\min\Defset{f(\vx)}{\vx \in Q_0 = \Omega}}$.
	This returns a solution $\incumbent = (x_1^k, x_2^k)$, for which $r$ is evaluated, and a radius of the norm-induced cut if the violation exceeds the accuracy.
	The corresponding numbers are summarized in~\Cref{tab:ex1}.
	
	\begin{table}
		\caption{Intermediate solutions and violations for \Cref{ex:nic}.}\label{tab:ex1}
		\begin{tabular}{cccc}
			\toprule
			$k$ & $(x_1^k, x_2^k)$ & $r(\incumbent)$ & $\abs{r(\incumbent)}/L$ \\ 
			\midrule
			$0$ & $(-1 ,-1)$ & $\approx 1.84$ & $\approx 1.30$ \\
			$1$ & $\approx(-0.08, -0.08)$ & $\approx 0.16$ & $\approx 0.11$\\
			$2$ & $\approx(4 \cdot 10^{-5}, 4 \cdot 10^{-5})$ & $\approx 8 \cdot 10^{-5}$ & -\\
			\bottomrule
		\end{tabular}
	\end{table}
	
	We illustrate the cutting procedure in \Cref{subfig:ex_nic_iter}.
	In the first iteration, a linear program is solved, resulting in $\vx^0$. 
	As observed, its distance to the feasible region can be considered large such that the first cut, which is illustrated as a dashed circle around its center $\vx^0$, shows a radius of similar magnitude.
	Solving the following reverse convex subproblem returns $\vx^1$. 
	The corresponding violation of $r$ is still beyond the accuracy such that another cut (dotted circle) is applied.
	After solving the subproblem for a third time to global optimality, the point $\vx^2$ is returned, which is sufficiently near the global optimal point and thus suffices for an $\eps$-optimal solution.
	
	\begin{figure}
		\centering
		\begin{subfigure}{0.48\textwidth}
				\centering
				\begin{tikzpicture}[scale=2]
						\clip (-1.4,-1.2) rectangle (1.4,1.2);
						% draw the coordinate system
						\draw[step=.5,gray,very thin] (-1.3, -1.2) grid (1.3, 1.2); %grid
						\draw[->] (-1.4, 0) -- (1.4, 0); %x1-axis
						\draw[->] (0, -1.2) -- (0, 1.2); %x2-axis
						\draw (-1, -2pt) node[anchor=north] {-1} -- (-1,2pt);
						\draw (1, -2pt) node[anchor=north] {1}-- (1, 2pt);
						\draw (-2pt, -1) -- (2pt,-1) node[anchor=west] {-1};
						\draw (-2pt, 1) -- (2pt, 1) node[anchor=west] {1};
						
						% draw -sin as a cos curve from -pi/2 to 0 and a sin curve from 0 to pi/2
						\draw[name path=sine, thick] (-1.57, 1) cos (0,0) sin (1.57, -1); 
						% draw the rest of the box starting at -sin(-1) ~ 0.84 to sin(1) ~-0.84
						\draw[thick] (-1,0.84) -- (-1, 1) -- (1, 1) -- (1, -0.84);
						
						% shade the feasible area
						\begin{scope}
								\clip (-1,-1) rectangle (1,1);
								\fill[pattern=north east lines] (-1.57, 1) cos (0,0) sin (1.57, -1) -- (1, 1) -- (-1, 1);
							\end{scope}
						
						\filldraw [black] (0,0) circle (1pt) node[below left, black] {$\sol$};
						
					\end{tikzpicture}
				\subcaption{Feasible region (hatched) and optimal solution $\sol = (0, 0)$ of \labelcref{example_nic}.}
				\label{subfig:ex_nic_feas}
			\end{subfigure}
		\hfill
		% draw the relaxed feasible region, the circles and the incumbents
		\begin{subfigure}{0.48\textwidth}
				\centering
				\begin{tikzpicture}[scale=2]
						\clip (-1.4,-1.2) rectangle (1.4,1.2);
						% draw the coordinate system
						\draw[step=.5,gray,very thin] (-1.3, -1.2) grid (1.3, 1.2); %grid
						\draw[->] (-1.4, 0) -- (1.4, 0); %x1-axis
						\draw[->] (0, -1.2) -- (0, 1.2); %x2-axis
						\draw (-1, -2pt) node[anchor=north] {-1} -- (-1,2pt);
						\draw (1, -2pt) node[anchor=north] {1}-- (1, 2pt);
						\draw (-2pt, -1) -- (2pt,-1) node[anchor=west] {-1};
						\draw (-2pt, 1) -- (2pt, 1) node[anchor=west] {1};
						
						% draw the box
						\draw[thick] (-1, -1) rectangle (1, 1);
						
						% draw the circles
						\draw[dashed] (-1, -1) circle [radius=1.30];
						\draw[dotted,thick] (-0.08, -0.08) circle [radius=0.11];
						
						% draw the incumbents
						\filldraw[black] (-1, -1) circle (1pt) node[above right, black] {$\vx^0$};
						\filldraw[black] (-.08, -.08) circle (1pt) node[below left, black] {$\vx^1$};
						\filldraw[red] (0, 0) circle (1pt) node[above right, black] {$\vx^2$};
					\end{tikzpicture}
				\subcaption{Incumbents $\vx^0$, $\vx^1$, $\vx^2$ with the respective cuts for \labelcref{alg:nic} on \labelcref{example_nic}.}
				\label{subfig:ex_nic_iter}
			\end{subfigure}
		\caption{Illustration of \Cref{ex:nic}.}
		\label{fig:ex_nic}
	\end{figure}
	
\end{example}

We note that the original problem in the example above cannot be treated by standard solvers such as Gurobi~\cite{Gurobi801} (before version 11) without reformulation.
The created subproblems, however, can be tackled by such methods, since the occurring constraints are quadratic ones.

\begin{example}[Necessity of global optimality]
	\label{ex:local}
	Here, we consider a one dimensional problem in $\reals$ with the absolute value~$\abs{\,\cdot\,}$ which reads
	\begin{equation}
		\label{example_local}
		\begin{aligned}
			\min\ & -\abs{x} \\
			\text{s.t.}\ & \frac13 x^3 \geq 0, \\
			& x \in [-1, 1].
		\end{aligned}
		\tag{$\mathrm{P}_\mathrm{bad}$}
	\end{equation}
	Following our notation, we have $f(x) = -\abs{x}$, $r(x) = -\frac13 x^3$, and $\Omega=[-1, 1]$. 
	Computing the Lipschitz constant according to~\Cref{subsubsec:L-constant-cont-diff} gives $L = 1$.
	Again, we omit the bold format and the subscript `+' due to the 1-dimensional $r$.
	
	In order to get a better understanding of problem~\labelcref{example_local}, we can rewrite it as $\max\Defset{x}{x \in [0, 1]}$ and see that its optimal point is at $x^* = 1$. 
	The ``worst possible'' point in terms of the optimization problem is considered $\hat{x} = 0$, as this is the maximum of problem~\labelcref{example_local}.
	We will show that \labelcref{alg:nic} converges to $\hat{x}$ if the subproblem in step~\labelcref{step:nic_solve} is only solved to local optimality, which is in contrast to the assumptions in~\Cref{thm:opt}.
	
	In this particular case, let's assume step~\labelcref{step:nic_solve} to always return a locally optimal solution in $[-1, 0)$ if it exists. 
	Therefore, relaxing the inequality $r(x) \leq 0$ in \labelcref{example_local}, we receive the solution $x^0 = -1$ of the corresponding subproblem with a violation of $r(x^0) = 1/3$.
	Adding the cut ${\abs{x - x^0} \geq r(x^0)/L = 1/3}$, we end up with $Q_1 = [-2/3, 1]$ as the feasible set of the next subproblem to investigate.
	We can show by induction 
	that the points produced by solving the subproblems satisfy the following relation
	\begin{equation}
		\label{ieq:ex_loc_iteration}
		x^{k+1} = x^k + r(x^k) = x^k - \frac13 (x^k)^3,
	\end{equation}
	for $k \in \naturals_0$.
	Using induction once more, we can prove that indeed $x^k \in [-1, 0)$ for all $k \in \naturals_0$.
	
	Now, we define the right-hand side of \labelcref{ieq:ex_loc_iteration} as the function $\nu(x) \coloneqq x - \frac13 x^3$.
	Note that this is a contraction on $[-1, 0]$.
	Therefore, we can use Banach's fixed-point theorem to see that $\nu$ has a unique fixed-point $\hat{x}$ on $[-1, 0]$.
	This fixed-point is indeed $\hat{x} = 0$, which can be seen by $\nu(0) = 0$.
	Further, Banach's theorem allows us to show that by starting with any $x^0 \in [-1, 0]$, the sequence defined by $x^{k+1} = \nu(x^k)$, which is equivalent to \labelcref{ieq:ex_loc_iteration}, converges to $\hat{x}$. 
	Therefore, $\lim_{k \rightarrow \infty} x^k = \hat{x} = 0$.
	
	In summary, this shows that the sequence produced by \labelcref{alg:nic} converges to a suboptimal point that can also be considered to be the ``worst possible'' point. 
	If we have a solution method returning a global optimal point, then the algorithm would return $x^* = 1$ in its second iteration at the latest.
	
\end{example}

We highlight that the feasibility of the limit is ensured independent of the optimality quality, as shown in~\Cref{thm:feas}.
Solving the subproblems to local optimality is only ensured to result in local optimal solutions if the algorithm stops in step~\labelcref{step:nic_ret_sol}.

Finally, we want to compare the performance when considering $\vr$ itself versus the components of $\vr$ separately with respect to the quality of available Lipschitz constants.
This was discussed in \Cref{rem:separate-cuts} and is made explicit here.

\begin{example}[Comparison of vector-valued versus component-wise constraints]
	\label{ex:berthold}
	We consider an example from the literature, in particular the continuous relaxation of Example~2.2 in~\cite{Berthold2014}.
	Following previous notation, we set 
	\begin{equation*}
		\vr(\vx) = \vr(x_1, x_2) = \begin{pmatrix} r_1(x_1, x_2) \\ r_2(x_1, x_2)\end{pmatrix} \coloneqq \begin{pmatrix}
			\cos(6x_1)/2 - x_2 + 1.8 \\
			-2\sin(4x_1)/\sqrt{x_1} + x_2 - 2
		\end{pmatrix},
	\end{equation*}
	which allows to give the problem in~\cite{Berthold2014} as
	\begin{equation}
		\label{example_berthold}
		\begin{aligned}
			\min_{x_1,\, x_2}\ & x_1 + 4x_2 \\
			\text{s.t.}\ & \vr(x_1, x_2) \leq \veczero,\\
			& x_1 \in [1, 10],\, x_2 \in [0, 4].
		\end{aligned}
		\tag{$\mathrm{P}_\mathrm{comp}$}
	\end{equation}
	As we want to make use of Gurobi~\cite{Gurobi801} in its ability as a non-convex quadratic solver for the subproblems, we tackle problem~\labelcref{example_berthold} with \ref{alg:nic} with respect to the Euclidean norm.
	In particular, squaring the latter leads to the cuts in the desired form for the solver.
	This is also applied when computing the Lipschitz constants in the way of~\Cref{subsubsec:L-constant-cont-diff}.
	The resulting (squared) values are $L_1^2 = 10$ and $L_2^2 \approx 42.83$ for $r_1$ and $r_2$, as well as $L^2  \approx 50.83$.
	
	Now, we leverage these values to relax constraint $\vr(x_1, x_2) \leq \veczero$ and apply norm-induced cuts of the form
	\begin{equation*}
		\norm[2]{\vx - \incumbent}^2 \geq \frac{\norm[2]{\vr(x^k_1, x^k_2)_+}^2}{L^2},
	\end{equation*}
	when considering $\vr$ as a vector-valued function, or 
	\begin{equation*}
		\norm[2]{\vx - \incumbent}^2 \geq \frac{\abs{r_1(x^k_1, x^k_2)}^2}{L_1^2}, \text{ and }
		 \norm[2]{\vx - \incumbent}^2 \geq \frac{\abs{r_2(x^k_1, x^k_2)}^2}{L_2^2},
	\end{equation*}
	if $r_1(x^k_1, x^k_2) > 0$ and $r_2(x^k_1, x^k_2) > 0$, respectively, when considering the components of $\vr$ separately.
	Running \ref{alg:nic} for 100 iterations with solving the subproblems to global optimality using Gurobi, we can compare the convergence of both alternatives to the optimal value in~\Cref{fig:ex-conv}.
	We note that the optimal objective value of Problem~\labelcref{example_berthold} is approximately $6.76$.
	
	\pgfplotstableread{data-2d-berthold.tex}{\tabletwo}
	\pgfplotstableread{data-1d-berthold.tex}{\tableone}
	\pgfplotstableread{data-2d-berthold_manip.tex}{\tabletwomanip}
	\pgfplotstableread{data-1d-berthold_manip.tex}{\tableonemanip}
	\begin{figure}
		\begin{subfigure}{0.49\textwidth}
			\centering
			\begin{tikzpicture}[scale=0.7]
				\begin{axis}[xmin = 0, xmax=100,xlabel=\textup{Iteration} ,ylabel=\textup{Objective},
					legend pos=south east]
					\addplot[mark=none, black, thick] coordinates {(0, 6.763847783176571) (99, 6.763847783176571)};
					\addlegendentry{\textup{opt. value}}
					\addplot[dashed] table [x = {iteration}, y = {objective value}] {\tableone};
					\addlegendentry{$r_1,\,r_2$\textup{-cuts}}
					\addplot[dotted, thick] table [x = {iteration}, y = {objective value}] {\tabletwo};
					\addlegendentry{$\vr$\textup{-cuts}}
				\end{axis}
			\end{tikzpicture}
			\caption{\labelcref{example_berthold} original.}
			\label{fig:ex-conv}
		\end{subfigure}
		\begin{subfigure}{0.49\textwidth}
			\centering
			\begin{tikzpicture}[scale=0.7]
				\begin{axis}[xmin = 0, xmax=100,xlabel=\textup{Iteration} ,ylabel=\textup{Objective},
					legend pos=south east]
					\addplot[mark=none, black, thick] coordinates {(0, 6.763847783176571) (99, 6.763847783176571)};
					\addlegendentry{\textup{opt. value}}
					\addplot[dashed] table [x = {iteration}, y = {objective value}] {\tableonemanip};
					\addlegendentry{$r_1,\,r_2$\textup{-cuts}}
					\addplot[dotted, thick] table [x = {iteration}, y = {objective value}] {\tabletwomanip};
					\addlegendentry{$\vr$\textup{-cuts}}
				\end{axis}
			\end{tikzpicture}
			\caption{\labelcref{example_berthold} manipulated.}
			\label{fig:ex-conv-manip}
		\end{subfigure}
		\caption{Objective value per iteration of \ref{alg:nic} on different versions of~\labelcref{example_berthold} when considering $\vr$ or its separate components.}
	\end{figure}
	
	One can observe that using the components of $\vr$ separately leads to faster convergence to the optimal objective value.
	In iteration 100, the relative difference of the optimal value of~\labelcref{example_berthold} to the current value of the relaxation~\labelcref{kth-relaxed-prob} is approximately 0.58\% when using $L_1$, $L_2$, and 18.55\% when using just $L$.
	This results from the additional information, \ie, tighter Lipschitz constants of the components.
	In particular, when only one constraint $r_1(\vx) \leq 0$ or $r_2(\vx) \leq 0$ is violated, considering them separately adds cuts with smaller Lipschitz constants, thus, a larger right-hand side than when considering only $\vr$.
	
	For the opposite observation in~\Cref{rem:separate-cuts}, consider the following manipulations to~\labelcref{example_berthold}.
	Let $r_1(\vx) = r_2(\vx) = \cos(6x_1)/2 - x_2 + 1.8$.
	The optimal objective value remains at approximately 6.76.
	For the Lipschitz constants, this gives $L^2 = 20$ as the smallest possible one.
	For $r_1$ and $r_2$, we assume to have a non-minimal Lipschitz constant $L_1^2 = L_2^2 = 15$ instead of 10 from above. 
	
	Applying the same procedure as above leads to the behavior illustrated in~\Cref{fig:ex-conv-manip}.
	Conversely to before, \labelcref{alg:nic} converges faster when considering $\vr$ as a whole, even though the effect is not as significant as before.
	For an interpretation, recall that informally the radius of a norm-induced cut is the violation divided by the Lipschitz constant.
	With this in mind, the violation of $\vr$ in the beginning appears to be larger than the difference of $L$ to $L_1$, $L_2$, leading to larger radii and equivalently an extended exclusion of the feasible region.
	As discussed in~\Cref{rem:separate-cuts}, the difference between considering $\vr$ or its components $r_1,\, r_2$ is thus dependent on the particular application and the availability of tight Lipschitz constants for each constraint function (component). 
	The Lipschitz constant might not even be available for the components of a vector-valued function, but only for the function itself, see~\Cref{subsubsec:L-constant-bilevel} for instance.

\end{example}

%% file: conclusion.tex
\section{Conclusion}
\label{sec:conclusion}

We have introduced the \labelcref{alg:nic} method in order to tackle the abstract problem class~\labelcref{original}.
Besides a continuous objective on a compact domain, the method requires a global Lipschitz constant for the constraint functions and an oracle to solve the subproblems~\labelcref{kth-relaxed-prob} globally.
Under these assumptions, we have proven the correctness of \labelcref{alg:nic} in terms of returning an optimal solution when terminating or,
if not, creating a sequence whose accumulation points are all optimal.
This has been complemented by discussions about the termination bound in the infeasible case and the problem complexity, showing asymptotic convergence of the bounds in the latter.
The theoretical content has been supplemented by examples, illustrating the method, visualizing the possible effects of a local oracle, and demonstrating the ambiguous effect of known Lipschitz constants for the components of the constraint function.

While there exist optimization software like Gurobi~\cite{Gurobi801} and GloMIQO~\cite{misener2013glomiqo}
which successfully tackle quadratically-constrained problems,
(until recently) they struggled when facing general nonlinear constraints.
For the special case of the Euclidean norm, our approach enables such software to solve even more general problems by executing \labelcref{alg:nic} using them as an oracle.

As observed especially in the examples, the magnitude of the radius of the norm-induced cuts depends on the Lipschitz constant $L$ in the denominator and is crucial. 
If an evaluation of the constraint function is cheap, one can compute an approximation of the point-dependent Lipschitz constant and, thus, reduce its value.
Apart from that, applications with small $L$ seem reasonable. 
For example, the field of optimal control of power networks
include sine or cosine constraints 
with small $L$ and, thus, could be tackled by \labelcref{alg:nic}.
This gives rise to another interesting topic of research.

%% file: acronyms.tex
\begin{acronym}[NIC]
	\acro{nic}[NIC]{Norm-induced Cutting Method}
	\acro{lhs}[LHS]{left-hand side}
\end{acronym}

%% file: acknowledgements.tex
\subsection*{Acknowledgments}
We thank the Deutsche Forschungsgemeinschaft for their
support within project A05 in the
\emph{Sonderforschungsbereich/Transregio 154 Mathematical
  Modelling, Simulation and Optimization using the Example of Gas
  Networks}.
  
Further, we thank Martin Schmidt for a fruitful discussion about the distinction to the present article.

%% file: appendix.tex
\appendix

\section{Reformulation of norm-induced cuts with 1-/$\infty$-norm}
\label{sec:reform-cuts}

If the norm-induced cuts \labelcref{eq:def-nic} are defined with respect to the 1-norm or the infinity norm (over $\RR$ or $\QQ$), then we can formulate the relaxed problem as a mixed-integer linear problem using binary variables. 
For simplicity, consider a cut of the form 
$$
\|\vx - \va\|_p \geq b,
$$
where $p=1$ or $\infty$.\\
\textbf{Case $p=1$:} The cut then is equivalent to 
	$$
	\sum_{i=1}^{n}|x_i - a_i| \geq b.
	$$
	By introducing continuous variables $y_i$ and binary variables $z_i$ this constraint can be equivalently expressed as 
	$$
	\begin{aligned}
		&\sum_{i=1}^{n} y_i \geq b, \\
		&y_i \geq (x_i - a_i) - M z_i,\\
		&y_i \leq (x_i - a_i) + M z_i,\\
		&y_i \geq -(x_i - a_i) - M (1-z_i),\\
		&y_i \leq -(x_i - a_i) + M (1-z_i),\\
		&z_i \in \{0,1\},\\
		&y_i \geq 0,
	\end{aligned}
	$$
	where all except for the first inequality are given for all $i \in [n]$ and $M > 0$ large.
	In order to show this equivalence of both models, on the one hand, observe that if $z_i = 0$ the respective constraints simplify to 
	$$
	\begin{aligned}
		&y_i \geq (x_i - a_i), \\
		&y_i \leq (x_i - a_i), \\
		&y_i \geq -(x_i - a_i) - M, \\
		&y_i \leq -(x_i - a_i) + M, 
	\end{aligned}
	$$
	for all $i \in [n]$, where the first two constraints ensure that $y_i = x_i - a_i$ and the second two constraints impose no restriction on $y_i$ for $M > 0$ sufficiently large.
	On the other hand, if $z_i = 1$, then we have 
	$$
	\begin{aligned}
		&y_i \geq (x_i - a_i) - M,\\
		&y_i \leq (x_i - a_i) + M,\\
		&y_i \geq -(x_i - a_i), \\
		&y_i \leq -(x_i - a_i),  
	\end{aligned}
	$$
	for all $i \in [n]$, so $y_i = -(x_i - a_i)$ is enforced.
	The last set of constraints $y_i \geq 0$ then ensures that variable $y_i$ is assigned a positive value, resulting in $y_i = \abs{x_i - a_i}$.
	In conclusion, $\sum_{j = 1}^n y_i \geq b$ is equivalent to the cut for the 1-norm.\\
\textbf{Case $p=\infty$:} The cut then is equivalent to
	$$
	\max_{i=1,\dots,n}|x_i - a_i| \geq b.
	$$
	In contrast to the first case, we have to introduce continuous variables $w_i$ and binary variables $u_i$ \textit{additionally} in order to cope with the maximum formulation. 
	Then the cut can be equivalently expressed as 
	$$
	\begin{aligned}
		&\sum_{i=1}^{n} w_i \geq b, \\
		&w_i \leq y_i + M(1 - u_i),\\
		&0 \leq w_i \leq M u_i,\\
		&\sum_{i=1}^{n}u_i = 1,\\
		&\sum_{j=1}^{n}w_j \geq y_i \\
		&y_i \geq (x_i - a_i) - M z_i,\\
		&y_i \leq (x_i - a_i) + M z_i, \\
		&y_i \geq -(x_i - a_i) - M (1-z_i),\\
		&y_i \leq -(x_i - a_i) + M (1-z_i),\\
		&z_i \in \{0,1\},\\
		&u_i \in \{0,1\},\\
		&y_i \geq 0,
	\end{aligned}
	$$
	for all $i \in [n]$.
	As in the case before, the variables $y_i$, $z_i$ and respective constraints ensure that $y_i = \abs{x_i - a_i}$.
	For the maximum, observe that $u_{i^*}$ is enforced to be non-zero, \ie, 1, for  exactly one $i^* \in [n]$. 
	This also implies that $u_i$ and due to the big-$M$ formulation $w_i$ are zero for $i \in [n]\setminus\{i^*\}$.
	Regarding $w_{i^*}$, we note that the upper bound supposed to be taking effect is by the second inequality, reducing to $w_{i^*} \leq y_{i^*}$.
	That is, exactly one $w_i$ is non-zero and upper bounded by $y_i$.
	In addition with $\sum_{j = 1}^n w_j \geq y_i$ for $i \in [n]$ we conclude that $w_{i^*}$ takes exactly the maximum of all $y_i$, all other $w_i$ being zero. 
	This shows the equivalence of modeling.
	
Note that the value of $M > 0$ can always be bounded by a diameter of the feasible region or an upper bound to this value, respectively.

\section{Examples for formula~\labelcref{ieq:volume}}
\label{sec:volume_ex}
Let $L > 0$ be the (upper bound to the) Lipschitz constant and $\delta > 0$ such that $\delta/L < 1$. 
We calculate the examples as appearing in~\Cref{tab:iter_bounds}. 
For the sake of clarity, we define the unit ball $B = B_1$ with respect to the considered norm.
\begin{itemize}
	\item[a)] Let $\Omega_1 = [\va, \vb]$ be an $n$-dimensional box with $\va, \, \vb \in \reals^n$ and $a_j < b_j$ for $j = 1,\dots,n$.
	We consider the maximum-norm, \ie, $\norm[\infty]{\vx} = \max_{i \in [n]} \abs{x_i}$ for $\vx \in \reals^n$ and $[n] \coloneqq \{1, \dots, n\}$.
	
	In this case, we have $B = [-1, 1]^n$ and, thus, $\vol(B) = 2^n$. 
	It follows, $\vol((\delta/L) B/2) = (\delta/L)^n$ and
	\begin{equation*}
		\Omega_1 + (\delta/L) B/2 = \bigtimes_{j = 1}^n \left[a_j - \frac{\delta}{2L},\, b_j + \frac{\delta}{2L}\right].
	\end{equation*}
	Therefore, $\vol(\Omega_1 + (\delta/L) B/2) = \prod_{j = 1}^n (b_j - a_j + \delta/L)$ and we conclude
	\begin{equation*}
		T = \frac{\vol(\Omega_1 + (\delta/L) B/2)}{\vol((\delta/L) B/2)} = \prod_{j = 1}^n \left[\frac{L}{\delta} (b_j - a_j) + 1\right].
	\end{equation*}

	\item[b)] Let us consider the maximum-norm again and $\Omega_2 = \Omega_1 \cap \integers^n$ to be the integral (lattice) points in $\Omega_1$.  
	Then, we can calculate the number of lattice points as 
	\begin{equation*}
		\abs{\Omega_2} = \prod_{j = 1}^n(\lfloor b_j \rfloor - \lceil a_j \rceil + 1).
	\end{equation*}
	Since we assumed $\delta/L < 1$ in the beginning, it holds that
	\begin{equation*}
		(\{\vect{z}_1\} + (\delta/L)B/2) \cap (\{\vect{z}_2\} + (\delta/L)B/2) = \emptyset, 
	\end{equation*}
 	for $\vect{z}_1, \, \vect{z}_2 \in \Omega_2, \, \vect{z}_1 \neq \vect{z}_2$.
	Therefore, we can derive the volume ${\vol(\Omega_2 + (\delta/L) B/2) = \abs{\Omega_2} \vol((\delta/L)B/2)}$ and, thus,
	\begin{equation*}
		T = \frac{\vol(\Omega_2 + (\delta/L) B/2)}{\vol((\delta/L)B/2)} = \frac{\abs{\Omega_2} \vol((\delta/L)B/2)}{\vol((\delta/L)B/2)} = \abs{\Omega_2}.
	\end{equation*}
	
	\item[c)] Here, we consider the Euclidean or 2-norm, \ie, $\norm[2]{\vx} = \sqrt{\left(\sum_{j = 1}^n x_j^2 \right)}$.
	For some radius $r > 0$, we have $B_r$ and can derive from standard literature that 
	\begin{equation*}
		\vol(B_r) = \frac{\pi^{n/2}}{\Gamma(\frac{n}{2} + 1)} r^n = \vol(B)r^n,
	\end{equation*}
	where $\Gamma(\cdot)$ denotes the gamma function.
	Now, for $D > 0$, let $\Omega_3 = B_D$ and it follows that
	\begin{equation*}
		\vol(\Omega_3 + (\delta/L)B/2) = \vol(B_{D + (\delta/(2L))}) = \vol(B) (D + (\delta/(2L)))^n.
	\end{equation*}
	Therefore, we can conclude
	\begin{equation*}
		T = \frac{\vol(B_D + (\delta/L)B/2)}{\vol((\delta/L)B/2)} = \frac{(D + (\delta/(2L)))^n}{(\delta/(2L))^n} = \left(\frac{2L}{\delta}D + 1\right)^n.
	\end{equation*}
\end{itemize}